\documentclass{article}
\usepackage[utf8]{inputenc}
\usepackage[a4paper, total={6in, 8in}]{geometry}

\usepackage{amsmath,amsfonts,amssymb}
\usepackage{amsthm} %proofs
\usepackage{mathtools}
\usepackage{enumitem}
\usepackage{dsfont} %indicator function

\usepackage{booktabs}
\usepackage{multirow}
\usepackage{hyperref}

\usepackage[font={it}]{caption}

\usepackage{subcaption}

\makeatletter
\newcommand{\newreptheorem}[2]{%
	\newenvironment{rep#1}[1]{%
		\def\rep@title{#2 \ref{##1}}%
		\begin{rep@proposition}}%
		{\end{rep@proposition}}}
\makeatother

\newtheorem{theorem}{Theorem}
\newtheorem{definition}{Definition}
\newtheorem{corollary}{Corollary}
\newtheorem{proposition}{Proposition}
\newreptheorem{proposition}{Proposition}
\newreptheorem{Theorem}{Theorem}
\newtheorem{lemma}{Lemma}
\newtheorem{remark}{Remark}
\newtheorem{example}{Example}

\DeclarePairedDelimiter{\Vpair}{\|}{\|} % norm like
\DeclarePairedDelimiter{\vpair}{|}{|} % abs like

\newcommand\E{\mathbb{E}} % symbol for expectation
\newcommand\norm[2][0pt]{\Vpair*{#2}} % use \norm{X}_2 for the length of vector X
\newcommand\abs[2][0pt]{\vpair*{#2}} % use \abs{X} for the absolute value of X

\DeclareMathOperator*{\argmin}{arg\,min}
\newcommand\real{\mathbb{R}} % \real for the symbol for real numbers

\usepackage{xcolor}
\newcommand\tcr[1]{} % Remove old text
\newcommand\tcb{}

\newcommand\xtrue{x^\star}

\newcommand\xa{ x^\alpha}

\newcommand\xak{x^{\alpha_k}}
\newcommand\ztrue{z^\star}

\newcommand\xo{x^0}
\newcommand\za{ z^\alpha}

\newcommand\ynoise{y}
\newcommand\datafit{\mathcal F}
\newcommand\reg{\mathcal R}
\newcommand\ulfun{\mathcal J}
\newcommand\llfun{\Phi}
\newcommand\op{A}

\newcommand\breg{D}
\newcommand\dini[1]{{#1}'_+}

\usepackage{authblk}
\usepackage{fancyhdr}
\title{On Optimal Regularization Parameters via Bilevel Learning}

\author{Matthias J. Ehrhardt}
\author{Silvia Gazzola}
\author{Sebastian J. Scott \thanks{Corresponding author: ss2767@bath.ac.uk}}
\affil{Department of Mathematical Sciences, University of Bath, Claverton Down, BA2 7AY, UK}

\makeatletter
\def\@maketitle{%
	\newpage
	\null
	\vskip 2em%
	\begin{center}%
		\let \footnote \thanks
		{\Large\bfseries \@title \par}%
		\vskip 1.5em%
		{\normalsize
			\lineskip .5em%
			\begin{tabular}[t]{c}%
				\@author
			\end{tabular}\par}%
	\end{center}%
	\par
	\vskip 1.5em}
\makeatother

\begin{document}
	\maketitle
	
	\begin{center}
		\Large \textbf{Abstract}
	\end{center}
		Variational regularization is commonly used to solve linear inverse
		problems, and involves augmenting a data fidelity by a regularizer.
		The regularizer is used to promote a priori information and is weighted by a regularization parameter.
		Selection of an appropriate regularization parameter is critical, with various choices leading to very different reconstructions. 
		Classical strategies used to determine a suitable parameter value include the  discrepancy principle and the L-curve criterion, and in recent years a supervised machine learning approach called bilevel learning has been employed. 
		Bilevel learning is a powerful framework to determine optimal parameters and involves solving a nested optimization problem. 
		While previous strategies enjoy various theoretical results, the well-posedness of bilevel learning in this setting is still an open question.
		In particular, a necessary property is positivity of the determined regularization parameter.
		In this work, we provide a new condition that better characterizes positivity of optimal regularization parameters than the existing theory.
		Numerical results verify and explore this new condition for both small and high-dimensional problems.
	
	\vspace{1em}

	\vspace{1em}
	
	\section{Introduction}	\label{sec:intro}
	
	Inverse problems are a class of mathematical problems where one is tasked to determine the input to a system given its output and some knowledge about the properties of said system. Such problems arise in many important science and engineering applications such as biomedical, astronomical, and seismic imaging  ~\cite{ArridgeSimon2019Sipu, BenningMartin2018Mrmf, EnglInverseBook,  HansenPerChristian2006Di:m}. 
	
	We consider the class of discrete linear inverse problems wherein we are interested in retrieving the exact input $\xtrue\in\real^n$ given a matrix $\op\in\real^{m\times n},$ and corrupted measurement $\ynoise\in\real^m$ satisfying
	\begin{equation}
		\ynoise \approx \op\xtrue. \label{eq:inverse-prob}
	\end{equation}
	
	The challenge with inverse problems such as \eqref{eq:inverse-prob} is that almost all interesting applications are ill-posed in the sense of Hadamard \cite{HadamardSurLP}, in that at least one of the following conditions regarding solutions is violated: existence; uniqueness; continuity with respect to the observed measurement $\ynoise.$ 
	A classical approach to remedy the ill-posedness of \eqref{eq:inverse-prob} is via variational regularization \cite{BenningMartin2018Mrmf,ChungJulianneLIPi}, wherein one solves a minimization problem such as
	\begin{equation}
		\min_{x\in\real^n} \left\{ \frac{1}{2}\norm{\op x - \ynoise}^2 + \alpha \reg(x)\right\}.\label{eq:var-reg-general}
	\end{equation}
	In \eqref{eq:var-reg-general} we consider the popular squared Euclidean distance data fidelity, which can be statistically motivated by considering the negative log likelihood of an additive Gaussian noise corruption of $\op\xtrue$ \cite{ArridgeSimon2019Sipu}.
	The regularizer  $\reg:\real^n\to\real$ is used to encourage a priori information of the ground truth $\xtrue$ in reconstructions. 
	\tcb{Although smooth regularizers are used \cite{BenningMartin2018Mrmf,NuytsJ.2002Acpp},} in recent years a popular non-smooth choice has been total variation (TV)  \cite{RudinLeonidI1992Ntvb} which encourages sharp edges in reconstructions. 
	\tcb{While TV has seen wide success, it is known to induce a staircasing artifact \cite{Ring2000SPSTVRP}.
		In an attempt to remedy this issue but capture the otherwise success of TV, an entire zoo of TV-like regularizers have been proposed \cite{BenningMartin2018Mrmf,bredies2020hotvag} which include those that build on higher-order derivatives such as second-order TV \cite{Papafitsoros2014SOTV}, incorporate non-local structure \cite{Gilboa2008NLTV}, infimal-convolution based approaches \cite{ChambolleAntonin1997ICTV}, and the total generalized variation \cite{bredies2010tgv}.
		Due to computational or theoretical reasons \cite{GazzolaLandanman2020KmfIP,SherryFerdia2020}, one may require a smooth approximation of such regularizers \cite{hubernorm,WolhbergRodriguez2007IRN}.
		We remark that there are a variety of other flavours of regularizers available, such as those that utilize convolutions \cite{ChenYunjin2014RLTo} or a trained neural network \cite{AmosBrandon2016ICNN,MukherjeeSubhadip2020Lcrf}.
	} 
	\tcr{While smooth regularizers are also used \cite{NuytsJ.2002Acpp}, one may be interested in a non-smooth regularizer \cite{ArridgeSimon2019Sipu,BenningMartin2018Mrmf}, such as TV, but, be it for computational or theoretical reasons \cite{GazzolaLandanman2020KmfIP,SherryFerdia2020}, require a smooth approximation instead, for which there are various approaches to achieve this \cite{hubernorm,WolhbergRodriguez2007IRN}.}
	The balance between the data fidelity and regularizer in \eqref{eq:var-reg-general} is controlled by the regularization parameter $\alpha\geq 0.$
	It is crucial to determine a suitable value of $\alpha,$ as a poor choice could lead to a noise dominated or oversmoothed reconstruction \cite{HansenPerChristian2010Dip}.

	There are a variety of existing techniques to determine an appropriate regularization parameter value, such as the discrepancy principle \cite{HansenPerChristian2010Dip}, generalized cross validation \cite{GolubGeneH1979GCaa}, or the L-curve criterion \cite{HansenPerChristian2010Dip}.
	\tcb{Depending on the choice of regularizer and algorithm used to solve \eqref{eq:var-reg-general}, efficient methods that yield appropriate parameter values can be achieved \cite{GazzolaLandanman2020KmfIP,GAZZOLA2014180,Langer2017APSTV,Shang2021RPSLRMR}.}
	In particular, there is no one-method-fits-all and rather each method works under different assumptions to varying degrees of success \cite{AnzengruberStephanW2010Mdpf,BAKUSHINSKII1984181,BoneskyThomas2009Mdpa,EnglInverseBook,GockenbachMarkS.2018Otco,HansenPerChristian2010Dip}.
	
	An alternative is machine learning, wherein an optimal parameter is found by minimizing some appropriate loss function.
	This can be achieved via bilevel learning - a popular data-driven approach to determine hyperparameters \cite{ArridgeSimon2019Sipu,CrockettCaroline2022BMfI,DelosReyesJC2016BPLf,KunischKarl2013ABOA} which sits in the wider class of bilevel optimization \cite{ColsonBenoît2007Aoob,SinhaAnkur2018ARoB}.
	In this work, we put emphasis on the following class of bilevel learning problems:

	\begin{subequations}
		\label{eq:bilevel}
		\begin{equation}
			\min_{\alpha\in\mathcal P} \left\{\ulfun(\alpha) := \frac{1}{2}\E\left[\norm{\xa(\ynoise) - \xtrue}^2\right]\right\} \label{eq:ul}
		\end{equation}
		\begin{equation}
			\text{subject to }\xa(\ynoise) = \argmin_{x\in\real^n} \left\{\llfun_\alpha(x) := \frac{1}{2}\norm{\op x - \ynoise}^2 + \alpha\reg(x)\right\}, \label{eq:ll}
		\end{equation}
	\end{subequations}
	where $\mathcal P$ is a parameter space \tcb{e.g. $\mathcal P = [0,\infty)$}\tcr{ and  we assume $\op$ and $\reg$ are such that the solution to \eqref{eq:ll} is unique - more details given later}.
	The expectation in \eqref{eq:ul} is the total expectation and, unless specified otherwise in which case it will be denoted by a subscript,  can be taken with respect to, say,  some underlying distribution of the ground truth or noise (that is, however $\op\xtrue$ is corrupted to yield $\ynoise$ in \eqref{eq:inverse-prob}).
	We do not require any properties of these distributions other than the expectations being well defined.
	Minimization problem \eqref{eq:ul} is referred to as the upper level problem, and \eqref{eq:ll} the lower level problem.
	We remark that we demand neither existence nor uniqueness of solutions to the bilevel learning problem \eqref{eq:bilevel}.

	We consider minimizing the expected squared error in \eqref{eq:ul} 
	which is the most popular choice of loss function in bilevel learning \cite{CrockettCaroline2022BMfI}, though other choices have been explored \cite{DelosReyesJC2016BPLf,GeipingJonasMoeller2019PMfD}.
	By minimizing the expected squared error, the determined parameter should perform well on average.

	While it is of theoretical importance to study the expected squared error upper level cost function,  usually in practice we instead have a finite number of training data $ (\xtrue_k , \ynoise_k),$ $k=1,\dots,K.$ 
	In this scenario the upper level cost function will be \tcb{the empirical risk,}
	$$ \frac{1}{2K}\sum_{k=1}^K \norm{\xa(\ynoise_k)  - \xtrue_k}^2,$$
	and the bilevel learning problem would be solved to determine a regularization parameter that performs well  \tcb{on the} training dataset.
	Then, given some unseen measurement data which is similar to the training data, we can expect  a solution to \eqref{eq:bilevel} to be a reasonable choice of parameter and the variational problem \eqref{eq:ll}
	can be solved. 
	
	% Multi parameter setting and how to solve the parameter
	Although the bilevel learning problem \eqref{eq:bilevel} is phrased to optimize over a scalar $\alpha$, the general framework extends to the multi-parameter setting, where $\alpha$ would instead be interpreted as a vector. For example: to find the weights of a sum of different regularizers \cite{DeLosReyesJ.C2016Tsoo}; the sampling of the forward operator for MRI \cite{SherryFerdia2020}; the weights in the field of experts model \cite{ChenYunjin2014RLTo}; the choice of norm for the data fidelity and regularizer \cite{Chung_EfficientLearningMethods_2022}; the parameters of an input-convex neural network acting as the regularizer \cite{AmosBrandon2016ICNN,MukherjeeSubhadip2020Lcrf}. In this work we are interested in bilevel learning as a regularization parameter choice rule, so remain in the scalar setting.

	Regarding how one solves a bilevel learning problem in general, if the parameter space $\mathcal P$ is low-dimensional, one can efficiently determine optimal parameters via search methods \cite{chunglearnreg,BergstraBengio2012RSfhO}. 
	In the general multi-parameter case, however,  a brute force search is not computationally feasible and other strategies must be employed.
	Many such approaches aim to calculate gradients of $\ulfun$ directly and use a gradient based approach to determine optimal parameters \cite{CrockettCaroline2022BMfI}. 
	Although the minimizer $\xa(\ynoise)$ is in general non-differentiable with respect to the associated parameters, provided the lower level cost function is sufficiently smooth, so-called hypergradients of $\ulfun$ can be derived using the implicit function theorem~\cite{KrantzStevenG2012TIFT,Tappen2009LOMEiC}.	
	An issue with this approach is that exact solutions of the lower level problem are required but are often computed numerically in practice. While results are still promising in spite of this \cite{SherryFerdia2020}, methods that acknowledge this inexactness have also been developed and studied \cite{EhrhardtMatthiasJ.2021IDOf,EhrhardtMatthiasJ2023AIHf,beyondbackprop}.
	While some approaches try to bypass the smoothness assumption on $\llfun_\alpha$ \cite{Ochs2016TechGradBasedBileOptm,Ochs2015BiloptiwithNonsmLower}, said  methods only consider an approximation of the original bilevel problem and so must be handled with care.

	% Lead into our setting and theory
	A critical theoretical issue of \eqref{eq:bilevel} is the well-posedness of the learning and the characterization of solutions, which can be used to inform the design of numerical methods.
	In recent years, literature has been developed to address these issues 
	\cite{DavoliElisa2023SCiN,DeLosReyesJ.C2016Tsoo,HollerGernot2018Abaf,KunischKarl2013ABOA}.
	One example is \cite{HollerGernot2018Abaf}, where  the considered parameter space is $\mathcal P = [\underline{\alpha} , \bar \alpha]$ for $0<\underline{\alpha}\leq\bar\alpha <\infty$ chosen a priori. In this setting it is possible to prove stability of the lower level problem and existence of solutions to the upper level problem under certain assumptions. 
	However, in imposing solutions lie in a  closed  interval bounded away from $0$ determined a priori, it is possible that for a given training dataset the determined parameter is suboptimal. 
	
	While a lot of existing theory and numerical methods for solving \eqref{eq:ll} explicitly assume that the regularization parameter is non-zero \cite{ChambolleAntonin2016Aitc,Chung_EfficientLearningMethods_2022,JinLorenzSchiffler2009ENrEsaASM}, which would be in alignment with the setting of \cite{HollerGernot2018Abaf},
	in allowing $\alpha$ to be zero, that is, to consider
	\begin{equation}
		\mathcal P = [0,\infty),\label{eq:p-choice}
	\end{equation}
	we will see that it is possible for $0$ to be a solution to \eqref{eq:bilevel} for even the most simple problems.
	We commit to the choice of \eqref{eq:p-choice} from now on.
	Depending on the application, the choice of $\alpha=0$ can lead to degeneracy of \eqref{eq:ll} such as non-uniqueness of minimizers. Furthermore, 
	for an optimal parameter to be $0$ may be an indication that the chosen regularizer is not well suited for the problem setting.
	Determining natural conditions which guarantee $0$ is not a solution to \eqref{eq:bilevel} is therefore crucial to not only exclude these degenerate cases, but also to characterize suitability of regularizers for a given application.
	Additionally, since the optimized parameter in \eqref{eq:bilevel} is the regularization parameter of a variational model, understanding when $0$ is not a solution will help establish the validity of bilevel learning as a mathematically sound regularization parameter choice strategy, since positivity of the determined parameter is a necessary property  \cite{EnglInverseBook}.
	Various works have contributed towards conditions that  guarantee $0$ is not a solution to \eqref{eq:bilevel} \cite{DavoliElisa2023SCiN,DeLosReyesJ.C2016Tsoo,KunischKarl2013ABOA} and, while primarily considering the pointwise denoising setting, wherein $\op=I$ and there is no expectation in \eqref{eq:ul}, have considered regularizers such as generalized Tikhonov \cite{KunischKarl2013ABOA}, TV-like and their \tcb{Huber} counterparts  \cite{DeLosReyesJ.C2016Tsoo}, and more recently a broad class of lower semicontinuous regularizers \cite{DavoliElisa2023SCiN}.
	
	To ease notation, we will sometimes suppress the dependence of the reconstruction on the observed measurement, that is, denote $\xa(\ynoise)$ simply as $\xa,$ depending on what we want to emphasize.

	%%%%%%%%%%%%%%%%%%%%%%%%%%%%%%%%%%%%%%%%%%%%%%%%%%%%%%%%%%%%%%%%%%%%
	%%%%%%%%%%%%%%%%%%%%%%%%%%%%%%%%%%%%%%%%%%%%%%%%%%%%%%%%%%%%%%%%%%%%
	\subsection{Our contribution}\label{sec:positivity}
	In this work we provide a new sufficient condition to guarantee that $0$ is not a solution to the bilevel learning problem \eqref{eq:bilevel} that is applicable to  inverse problems with an injective forward operator, rather than just the denoising setting. 
	In particular, the new condition ensures that a generalized directional derivative of $\ulfun$ at $0$ is strictly negative for the class of lower level cost functions $\llfun_\alpha$ for which the regularizer $\reg$ is real-valued, bounded below, convex, and continuously differentiable.
	We also provide an extension to the setting where the upper level cost function is the expected predicted risk and the forward operator is invertible.
	Full statements of the main results are provided in Section \ref{sec:main} and, after some preliminary results stated in Section \ref{sec:ll-properties}, are proved in Section \ref{sec:main-result-proof}.

	We show that, in the pointwise setting, data $(\xtrue,\ynoise)$ that satisfy a condition commonly used to deduce positivtiy of optimal parameters \cite{DavoliElisa2023SCiN,DeLosReyesJ.C2016Tsoo}  immediately satisfy our new condition, which is empirically shown in Section \ref{subsec:numerics} to be a better characterization of positivity.
	In allowing an expectation in the upper level problem, our setting is in contrast to existing work \cite{DavoliElisa2023SCiN,DeLosReyesJ.C2016Tsoo,KunischKarl2013ABOA}  where only a pointwise problem is considered.
	We show not only  will our condition always be satisfied in a realistic denoising setting, but it can completely characterize positivity for certain applications.

	%%%%%%%%%%%%%%%%%%%%%%%%%%%%%%%%%%%%%%%%%%%%%%%%%%%%%%%%%%%%%%%%%%%%
	%%%%%%%%%%%%%%%%%%%%%%%%%%%%%%%%%%%%%%%%%%%%%%%%%%%%%%%%%%%%%%%%%%%%
	%%%%%%%%%%%%%%%%%%%%%%%%%%%%%%%%%%%%%%%%%%%%%%%%%%%%%%%%%%%%%%%%%%%%
	%%%%%%%%%%%%%%%%%%%%%%%%%%%%%%%%%%%%%%%%%%%%%%%%%%%%%%%%%%%%%%%%%%%%
	%%%%%%%%%%%%%%%%%%%%%%%%%%%%%%%%%%%%%%%%%%%%%%%%%%%%%%%%%%%%%%%%%%%%
	%%%%%%%%%%%%%%%%%%%%%%%%%%%%%%%%%%%%%%%%%%%%%%%%%%%%%%%%%%%%%%%%%%%%
	%%%%%%%%%%%%%%%%%%%%%%%%%%%%%%%%%%%%%%%%%%%%%%%%%%%%%%%%%%%%%%%%%%%%
	%%%%%%%%%%%%%%%%%%%%%%%%%%%%%%%%%%%%%%%%%%%%%%%%%%%%%%%%%%%%%%%%%%%%
	\section{Main results}\label{sec:main}
	The choice of regularizer is problem specific, as what constitutes as a suitable reconstruction varies between applications. 
	In general, $\reg$ should attain a large evaluation for an $x$ that exhibits undesirable properties.
	For the denoising application we have $\op=I,$ and are trying to improve upon the noisy measurement $\ynoise.$ In particular, $\reg$ should deem $\ynoise$ less desirable than the ground truth $\xtrue.$
	Thus, it is natural to assume that 
	\begin{equation}
		\reg(\xtrue) < \reg(\ynoise).\label{eq:heuristic-reg-cond}
	\end{equation}  
	Indeed, \eqref{eq:heuristic-reg-cond} has been considered in recent works \cite{DavoliElisa2023SCiN,DeLosReyesJ.C2016Tsoo} to deduce positivity of solutions of the bilevel learning problem.

	While we do not demand it here, typically the regularizer $\reg$ involves a norm and in particular is an even function.
	Consequently, for a fixed $\xtrue$, \eqref{eq:heuristic-reg-cond} is inherently circular around the origin and as such may fail to fully characterize the $\ynoise$ for which $0$ is not a solution to \eqref{eq:bilevel}. 
	While our condition will encompass applications with  an injective forward operator, to give an intuition of how it compares to \eqref{eq:heuristic-reg-cond}, in the denoising setting our condition will read as requiring the linearization of $\reg$ around $\ynoise$ evaluated at $\xtrue$ to be smaller than $\reg(\ynoise)$ to conclude that $0$ is not a solution to \eqref{eq:bilevel}.
	To represent this linearization, we will find it useful to work with Bregman distances, which are defined as follows.
	\begin{definition}[Bregman distance]
		For a differentiable convex function $\psi:\real^n\to \real$ with gradient $\nabla \psi$, the Bregman distance is defined as
		\begin{eqnarray*}
			\breg_\psi (x, \tilde x) := \psi(x) - \psi(\tilde x) - \langle  \nabla \psi(\tilde x) , x-\tilde x\rangle.
		\end{eqnarray*}
	\end{definition}
	Bregman distances can be considered a generalization of the squared Euclidean norm, and have nice properties such as  convexity in the first argument and non-negativity \cite{BurgerMartin2015BDiI}.
	We remark that while some definitions demand $\psi$ be strictly convex \cite{Censor1981}, we do not require that here as convexity provides all the properties we need for the scope of this work.
	We now give a few examples of Bregman distances.
	\begin{example}\label{ex:bregman-distances1} 
		Let $\reg(x) =\frac{1}{2}\norm{L x}^2 $ where $L\in\real^{p \times n}.$ One can show that the Bregman distance is given by
		\begin{equation*}
			\breg_\reg (x,z) = \frac{1}{2}\norm{L (x - z)}^2.
		\end{equation*}
		For the case $n=1, L=I,$ a graph of $\reg$ and Bregman distances $\breg_\reg(\cdot,z)$ for different values of $z$ is provided in Figure~\ref{fig:example-breg-tikh}.
		\begin{figure}
			\centering
			\includegraphics[width=.5\textwidth]{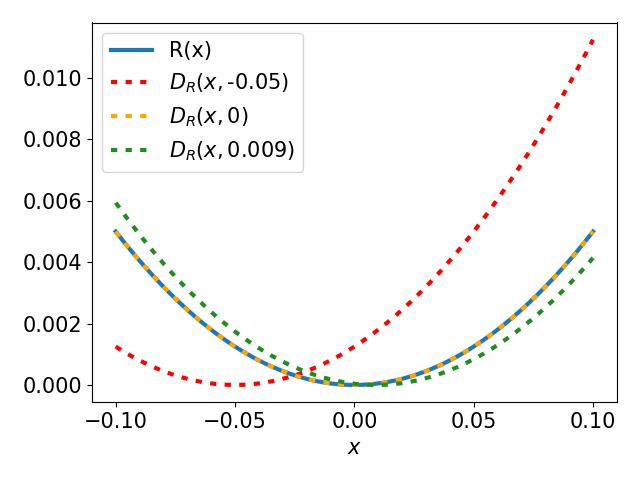}
			\caption{The case $n=1,$  $\reg(x) = \frac{1}{2}\norm{x}^2.$  
				Plot of regularizer evaluations and Bregman distances $\breg_\reg(\cdot,z)$ for different choices of $z$.}   
			\label{fig:example-breg-tikh}
		\end{figure}
	\end{example}
	
	\begin{example}\label{ex:bregman-distances2}    
		Let $n=1$ and  $\reg(x) = \mathrm{hub}_\gamma( x)$ 
		where 
		\begin{equation*}
			\mathrm{hub}_\gamma(x) = \left\{ \begin{array}{ll}
				\abs{x} - \frac{\gamma}{2} & \text{if } \abs{x}\geq \gamma
				\\
				\frac{1}{2\gamma} x^2 & \text{if } \abs{x} < \gamma.
			\end{array}\right.
		\end{equation*}
		One can show that the Bregman distance is given by 
		\begin{equation*}
			\breg_\reg(x,z) =  
			\left\{
			\begin{array}{ll}
				\left(\mathrm{sign}(x) - \mathrm{sign}(z)\right) x 
				& \text{if } \abs{x}\geq \gamma, \abs{z}\geq\gamma
				\\
				x \left(\mathrm{sign}(x) - \frac{1}{\gamma} z\right) +\frac{1}{2\gamma}z^2- \frac{\gamma}{2}
				& \text{if } \abs{x}\geq \gamma, \abs{z}<\gamma
				\\
				x \left(\frac{1}{2\gamma}x - \mathrm{sign}(z)\right) + \frac{\gamma}{2}
				& \text{if } \abs{x} < \gamma, \abs{z}\geq\gamma
				\\
				\frac{1}{2\gamma}\left(x-z\right)^2 
				& \text{if } \abs{x}< \gamma, \abs{z}<\gamma
			\end{array}
			\right.
		\end{equation*}
		For the case $\gamma=0.01,$ a graph of $\reg$ and Bregman distances $\breg_\reg(\cdot,z)$ for different values of $z$ is provided in Figure~\ref{fig:example-breg-hub}.
		
		\begin{figure}
			\centering
			\includegraphics[width=.5\textwidth]{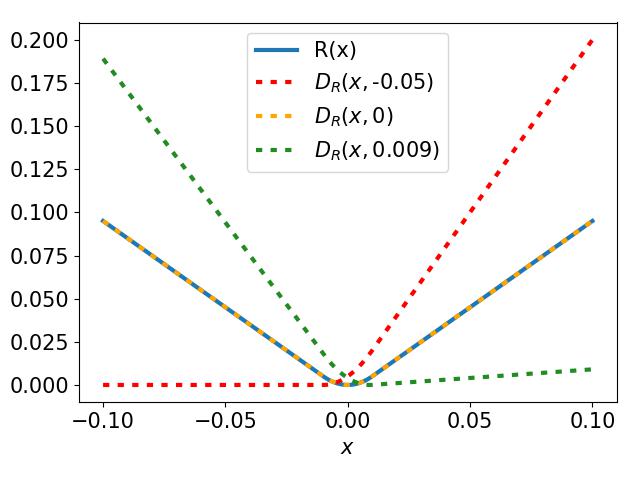}
			\caption{The case $n=1,$  $\reg(x) = \mathrm{hub}_\gamma(x),$ $\gamma=0.01.$  
				Plot of regularizer evaluations and Bregman distances $\breg_\reg(\cdot,z)$ for different choices of $z$.}   
			\label{fig:example-breg-hub}
		\end{figure}
	\end{example}
	
	From the definition of the Bregman distance, it is clear that, when viewed as a function of $x,$ it represents the distance between $\psi(x)$ and the linearization of $\psi$ around $\tilde x$ evaluated at $x.$
	
	Rather than requiring $\op=I$ and \eqref{eq:heuristic-reg-cond}, we merely require that $\op$ is injective and
	\begin{equation}
		\reg(B \xtrue) - \breg_\reg (B \xtrue,\xo) < 
		\reg(B \xo) - \breg_\reg (B \xo,\xo),
		\label{eq:our-theory}
	\end{equation}
	where $B:= (\op^T\op)^{-1},$
	to deduce that $0$ is not a solution to \eqref{eq:bilevel}.
	A full statement of the result is provided in Theorem~\ref{thm:expec}.
	Since $\op$ is  injective, $\op^T\op$ is invertible and $\xo$ is the unique least-squares solution and so \eqref{eq:our-theory} is well-defined.
	An injective forward operator captures various relevant applications, such as certain convolutions and the Radon transform \cite{Markoe2006RadonTransform,Gestur2006IXtRt}.
	If we are in the denoising setting and \eqref{eq:heuristic-reg-cond} is satisfied, we immediately get that \eqref{eq:our-theory} is also satisfied by the non-negativity of the Bregman distance.
	Figure~\ref{subfig:theory_better} compares, for a fixed $\xtrue$ and denoising application, how the region of possible $\xo$ for which  \eqref{eq:heuristic-reg-cond} and  \eqref{eq:our-theory} are satisfied differ.
	Figure~\ref{subfig:bregman} provides a geometric interpretation of the new condition for a non-denoising application.
	\begin{figure}
		\centering
		\begin{subfigure}[b]{0.45\textwidth}
			\centering
			\includegraphics[width=\textwidth]{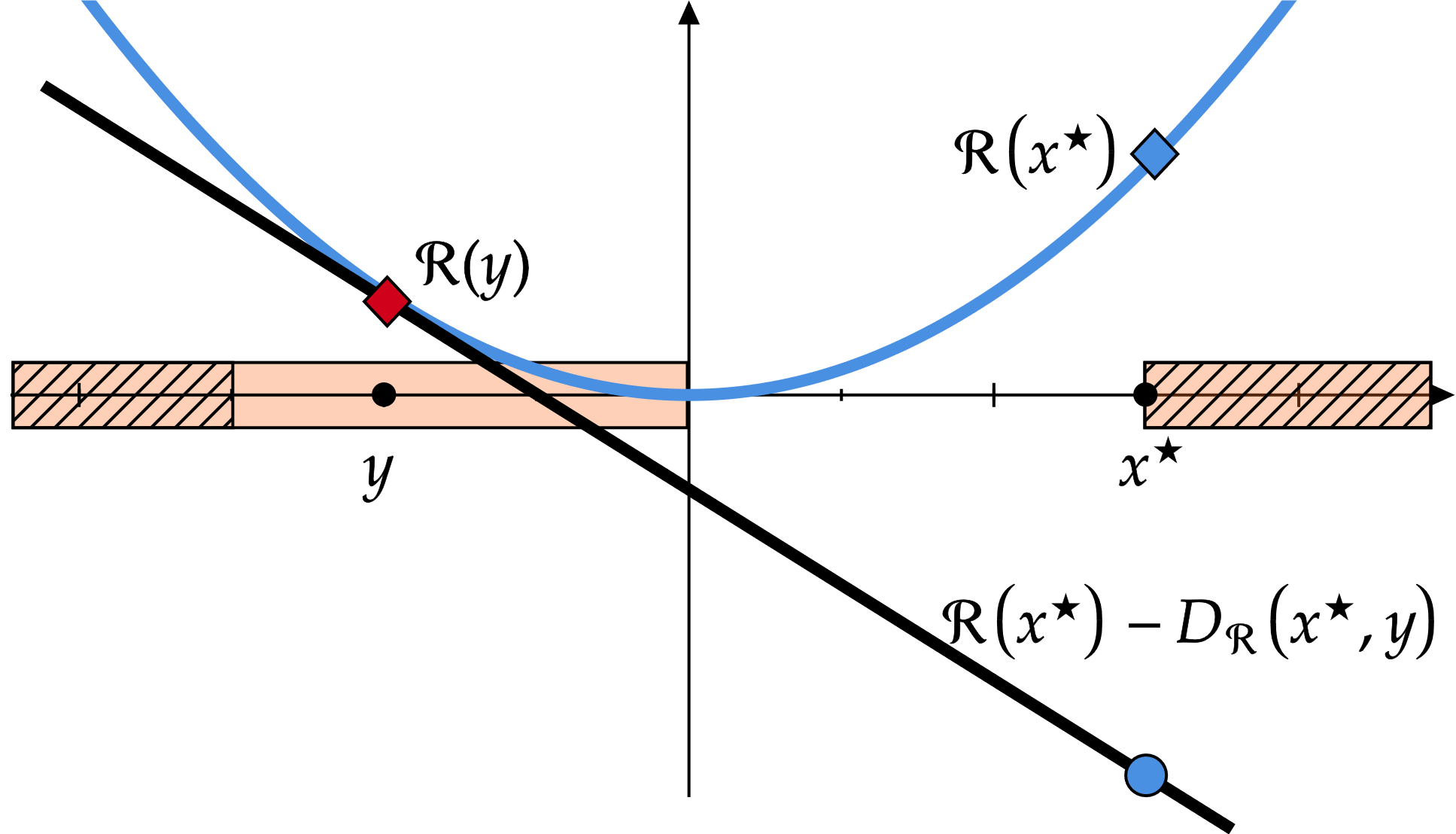}
			\caption{}
			\label{subfig:theory_better}
		\end{subfigure}
		\hfill
		\begin{subfigure}[b]{0.45\textwidth}
			\centering
			\includegraphics[width=\textwidth]{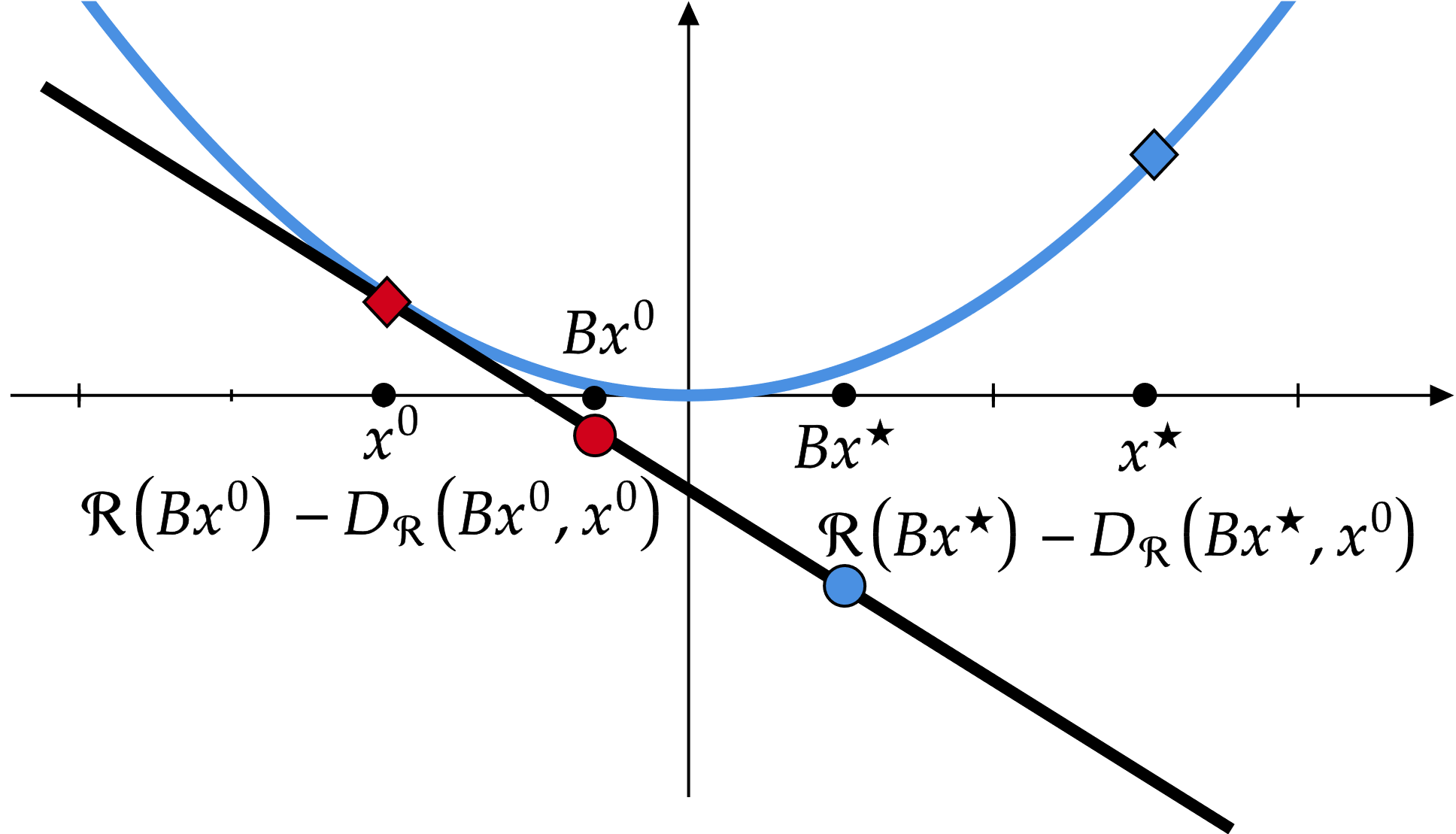}
			\caption{}
			\label{subfig:bregman}
		\end{subfigure}
		\caption{The case $n=1,$  $\reg(x) = \frac{1}{2}\norm{x}^2.$  
			\tcb{(a) shows a} plot of regularizer evaluations and illustration of the new condition in the case $A=1.$ While the old condition demands that the red diamond is higher than the blue diamond, the new condition only requires that the red diamond is higher than the blue circle. Regions of $y$ for which the old and new conditions are satisfied are indicated by the  striped and orange regions on the horizontal axis respectively. We see that the two conditions describe different regions and the striped region is a subset of the shaded region.
			\tcb{(b) shows a} plot of regularizer evaluations and illustration of the new condition in the case $A=\sqrt{3}.$ The new condition only demands that the red circle is higher than the blue circle.}   
		\label{fig:bregman}
	\end{figure}

	%%%%%%%%%%%%%%%%%%%%%%%%%%%%%%%%%%%%%%%%%%%%%%%%%%%%%%%%%%%%%%%%%%%%%%%%%%%%%%%%
	We now provide a motivating example illustrating that \eqref{eq:our-theory} can completely characterize  whether $0$ is a solution to \eqref{eq:bilevel}.
	\begin{example}\label{ex:perfect-characterization}
		Consider the denoising setting with Tikhonov regularization, $\reg = \frac{1}{2}\norm{\cdot}^2,$ and a single data sample.  In this setting the lower level solution is given analytically as $\xa = \ynoise/(1+\alpha).$
		Further, assume that $\norm{\ynoise}\neq 0,$ $\langle\ynoise,\xtrue\rangle >0,$ and a solution to \eqref{eq:bilevel} exists.
		We claim that $0$ is not a minimum if and only if $\eqref{eq:our-theory}$ is satisfied.
		
		Now, the solution to \eqref{eq:bilevel} will either be at the boundary or interior of $\mathcal P.$ Evaluation of the upper level at the boundary yields
		\begin{equation*}
			\ulfun(0) = \frac{1}{2}\norm{\ynoise - \xtrue}^2.
		\end{equation*}
		Optimal solutions $\bar\alpha$ in the interior will satisfy $0=\ulfun'(\bar\alpha),$ from which one can show that 
		\begin{equation}
			\bar \alpha = \frac{\norm{\ynoise}^2}{\langle \ynoise,\xtrue\rangle} - 1,\label{eq:example-param}
		\end{equation}
		which is not strictly positive in general.
		
		Suppose $0$ is not a minimum of $\ulfun$. Then  $\ulfun(0)>\ulfun(\bar\alpha)$ with $\bar\alpha>0,$ and it follows from \eqref{eq:example-param} that  $\bar\alpha>0$  if and only if
		\begin{align*}
			\frac{1}{2}\norm{\xtrue}^2 - \frac{1}{2}\norm{\ynoise - \xtrue}^2 < 
			\frac{1}{2}\norm{\ynoise}^2,
		\end{align*}
		that is, \eqref{eq:our-theory} is satisfied.  
		
		Suppose \eqref{eq:our-theory} is satisfied. 
		By the above this is equivalent to $\bar \alpha>0.$ 
		To deduce that $0$ is not a minimum, it remains to show that $\ulfun(0)>\ulfun(\bar\alpha).$
		Indeed, the associated upper level cost of $\bar\alpha$ is
		\begin{equation*}
			\ulfun(\bar\alpha) = \frac{1}{2}\norm{\xtrue}^2 - \frac{1}{2}\frac{\langle\ynoise,\xtrue\rangle^2}{\norm{\ynoise}^2}
		\end{equation*}
		and, since $\bar\alpha>0,$  it follows from \eqref{eq:example-param} that 
		\begin{align*}
			\frac{1}{2}\left(\frac{\langle\ynoise,\xtrue\rangle}{\norm{\ynoise}} - \norm{\ynoise} \right)^2
			&> 0 \\
			\iff 
			\frac{1}{2}\norm{\ynoise - \xtrue}^2
			&> \frac{1}{2}\norm{\xtrue}^2
			- 
			\frac{1}{2}\frac{\langle\ynoise,\xtrue\rangle^2}{\norm{\ynoise}^2},	
		\end{align*}
		that is, $\ulfun(0)>\ulfun(\bar \alpha) .$
		
		To summarize, for a denoising problem with  Tikhonov regularization,  \eqref{eq:our-theory} is satisfied if and only if $0$ is not a solution to \eqref{eq:bilevel}.
	\end{example}

	%%%%%%%%%%%%%%%%%%%%%%%%%%%%%%%%%%%%%%%%%%%%%%%%%%%%%%%%%%%%%%%%%%%%%%%%%%%%%
	%%%%%%%%%%%%%%%%%%%%%%%%%%%%%%%%%%%%%%%%%%%%%%%%%%%%%%%%%%%%%%%%%%%%%%%%%%%%%
	%%%%%%%%%%%%%%%%%%%%%%%%%%%%%%%%%%%%%%%%%%%%%%%%%%%%%%%%%%%%%%%%%%%%%%%%%%%%%
	%%%%%%%%%%%%%%%%%%%%%%%%%%%%%%%%%%%%%%%%%%%%%%%%%%%%%%%%%%%%%%%%%%%%%%%%%%%%%
	
	The main result is the following.
	
%	\pagebreak
	\begin{theorem}[Positivity of the bilevel learning problem solution]\label{thm:expec}
		Let $\op$ be injective and let $\reg$ be convex, bounded below, and continuously differentiable. If 
		\begin{equation} 
			\E \left[ \reg(B\xtrue) - \breg_\reg(B\xtrue,\xo(\ynoise))\right] <
			\E \left[	\reg(B\xo) - \breg_\reg(B\xo(\ynoise),\xo(\ynoise))\right] \label{eq:main-assumn-expec-og}
		\end{equation}		
		and
		\begin{equation*}
			\E\left [	\reg(B\xtrue) + \breg_\reg(B \xo(\ynoise),\xo(\ynoise))  \right] < \infty,
		\end{equation*}		
		then $0$ is not a solution to \eqref{eq:bilevel}.
	\end{theorem}	
	
	The proof is provided in Section \ref{sec:main-result-proof}.
	
	\tcb{Since Theorem~\ref{thm:expec} can be considered a characterization of whether a regularizer is appropriate for a given application, ensuring \eqref{eq:main-assumn-expec-og} is satisfied could inform the design of new, or augmentation of existing regularizers for specific applications.
		Furthermore, we will later see that the difference between both sides of \eqref{eq:main-assumn-expec-og} will correspond to the steepness of a generalized directional derivative of $\ulfun$ at $0.$ Thus selecting $\reg$ such that not only is \eqref{eq:main-assumn-expec-og} satisfied, but}
	\begin{equation*}\tcb{
			\E\left[\reg(B \xtrue) - \breg_\reg (B \xtrue,\xo)\right] \ll 
			\E\left[\reg(B \xo) - \breg_\reg (B \xo,\xo)\right]}
	\end{equation*}
	\tcb{can be an indication that the choice of regularizer  will yield noticeable improvement of reconstruction quality compared to the unregularized reconstruction.}

	We now state an immediate result of Theorem~\ref{thm:expec} for the pointwise setting, wherein $(\xtrue,\ynoise)$ is fixed and so there will be no expectation in the upper level \eqref{eq:ul}.
	
	\begin{corollary}
		\label{cory:pointwise}
		Let $\op$ be injective and let $\reg$ be convex, bounded below, and continuously differentiable.  
		If $\xtrue\in\real^n$ and $ \ynoise\in\real^m$ are fixed and 
		\begin{equation*}
			\reg(B \xtrue) - \breg_\reg (B \xtrue,\xo) < 
			\reg(B \xo) - \breg_\reg (B \xo,\xo),
		\end{equation*}
		then $0$ is not a solution to  \eqref{eq:bilevel}.
	\end{corollary}
	
	\begin{remark}
		For the special case of $\op=I$ and $\reg(x) = \frac{1}{2}\norm{Lx}^2$ where $L\in\real^{p\times n},$ Corollary~\ref{cory:pointwise} is proven in \cite[Proposition~3.1]{KunischKarl2013ABOA} where an equivalent condition, $\langle L^TL \ynoise, \xtrue - \ynoise\rangle>0,$  is assumed. 
	\end{remark}
	
	Another corollary of Theorem~\ref{thm:expec} is that, for a realistic denoising problem, we are guaranteed that $0$ is not a solution to \eqref{eq:bilevel}.
	\begin{corollary}\label{cor:sym-breg}
		Let $\op=I$ and let $\reg$ be convex, bounded below, and continuously differentiable.
		If $\xtrue$ is fixed and $\ynoise = \xtrue + \epsilon$ where $\E_\epsilon[\epsilon]=0$, then \tcb{we have the following.}
		\begin{enumerate}[label=(\roman*)]
			\item \tcb{Condition \eqref{eq:main-assumn-expec-og} is equivalent to}
			\begin{equation}
				0< \E_{\epsilon}\left[
				\breg_R(\ynoise,\xtrue) + \breg_\reg(\xtrue,\ynoise)\label{eq:sym-breg}
				\right].
			\end{equation}
			\item \tcb{If $\reg$ is strictly convex, then $0$ is not a solution to \eqref{eq:bilevel} almost surely.}
		\end{enumerate}						
	\end{corollary}
	The proof of Corollary~\ref{cor:sym-breg} is provided in Section~\ref{sec:main-result-proof}.
	
	\begin{remark}
		\tcb{    The quantity in the expectation in \eqref{eq:sym-breg} is called the symmetric Bregman distance. For details regarding properties and use of the symmetric Bregman distance in optimisation see, for example, \cite{Benning2017GDIBDF,BurgerM.2007EefB}.}
	\end{remark}
	
	Now the bilevel learning problem \eqref{eq:bilevel}, and thus the result of Theorem~\ref{thm:expec},  is specific to the expected squared error upper level cost, 
	\tcb{where access to clean samples $\xtrue$ is required.
		If direct access to $\xtrue$ is not be possible,  one could instead consider a cost function based on the measurement space \cite{DeledalleSUGAR,Fehrenbach2015BIDUGT, Hintermuller2022DADPSTGVBL,Sixou2021Arppnba,Zhang2020BLNSOAMFFD}.
		For this reason} we will also consider an alternative upper level cost function, namely, the expected predictive risk, 
	\begin{equation*}
		\ulfun(\alpha) = \E \left[ \frac{1}{2} \norm{\op \xa- \op\xtrue}^2
		\right],
	\end{equation*}
	where we will require that $\op$ is invertible.
	\tcb{
		We leave the consideration of other cost functions as future work.}
	In this setting, the associated bilevel learning problem is
	\begin{subequations}\label{eq:bilevel-pred-risk-og}
		\begin{equation}
			\min_{\alpha\in \mathcal P} \E\left[ \frac{1}{2}\norm{\op\xa-\op\xtrue}^2\right],\label{eq:bilevel-pred-risk-ul-og}
		\end{equation}
		\begin{equation}
			\text{subject to }	\xa = \arg\min_{x\in\real^n} \left\{\frac{1}{2} \norm{\op x-\ynoise}^2 + \alpha \reg(x)  \right\} .
		\end{equation}
	\end{subequations}
	Using Theorem~\ref{thm:expec} and the invertibility of $\op$, positivity of solutions  of the bilevel learning problem \eqref{eq:bilevel-pred-risk-og} can be established.
	
	\begin{theorem}\label{thm:positivity-deblur}
		Let $\op$ be invertible and let $\reg$ be convex, bounded below, and continuously differentiable. If  
		\begin{equation*}
			\E\left[	\reg( \xtrue) - \breg_\reg(\xtrue,\xo(\ynoise) )\right] <
			\E\left[\reg(\xo(\ynoise))\right]
		\end{equation*} and \begin{equation*}
			\E\left[\reg( \xtrue)\right] <\infty,
		\end{equation*}
		then $0$ is not a solution to \eqref{eq:bilevel-pred-risk-og}.
	\end{theorem}
	The proof of Theorem \ref{thm:positivity-deblur} is provided in Section~\ref{sec:main-result-proof}. 
	We have an immediate corrollary for the pointwise setting, wherein there is no expectation in the upper level \eqref{eq:bilevel-pred-risk-ul-og}.
	\begin{corollary}
		\label{cory:pointwise-deblur}
		Let $\op$ be invertible and let $\reg$ be convex, bounded below, and continuously differentiable.  
		If $\xtrue\in\real^n$ and $ \ynoise\in\real^m$ are fixed and 
		\begin{equation*}
			\reg( \xtrue) - \breg_\reg ( \xtrue,\xo) < 
			\reg( \xo) ,
		\end{equation*}
		then $0$ is not a solution to  \eqref{eq:bilevel-pred-risk-og}.
	\end{corollary}

	Before we prove the main results we first cover in Section \ref{sec:ll-properties} some fundamental results and properties of the lower level problem.

	%%%%%%%%%%%%%%%%%%%%%%%%%%%%%%%%%%%%%%%%%%%%%%%%%%%%%%%%%%%%%%%%%%%%%%%%%%%%%%%%
	%%%%%%%%%%%%%%%%%%%%%%%%%%%%%%%%%%%%%%%%%%%%%%%%%%%%%%%%%%%%%%%%%%%%%%%%%%%%%%%%
	%%%%%%%%%%%%%%%%%%%%%%%%%%%%%%%%%%%%%%%%%%%%%%%%%%%%%%%%%%%%%%%%%%%%%%%%%%%%%%%%
	%%%%%%%%%%%%%%%%%%%%%%%%%%%%%%%%%%%%%%%%%%%%%%%%%%%%%%%%%%%%%%%%%%%%%%%%%%%%%%%%
	
	\section{Preliminaries}\label{sec:ll-properties}
	We start by stating relevant definitions and properties of the lower level cost function.
	In particular, we require properties that are sufficient for existence and uniqueness of solutions to the lower level problem, as well as continuity of reconstructions with respect to the regularization parameter.
	The following definitions are taken from \cite{BeckAmir2017Fmio} and \cite{ChambolleAntonin2016Aitc}.

	\begin{definition}[Minimizer] We say that $\hat x$ is a global minimizer of $\psi:\real^n \to \real$ if	$\psi(\hat x) \leq \psi(x)$ for all $x\in\real^n.$
	\end{definition}
	
	\begin{definition}[Bounded below] A function $\psi:\real^n\to \real$ is bounded below if there exists $C\in\real$ such that 
		\begin{equation*}
			\psi(x) \geq C\quad\text{for all }x\in\real^n.
		\end{equation*}
	\end{definition}

	\begin{definition}[Coercive] A function $\psi:\real^n\to \real$ is coercive if
		\begin{equation*}
			\psi(x) \to\infty\qquad\text{as }\norm{x}\to\infty.
		\end{equation*}
	\end{definition}

	\begin{definition}[Convex function] A function $\psi:\real^n\to \real$ is said to be convex if for all $x,z\in\real^n,$
		\begin{equation}
			\breg_\psi(x,z)\geq 0
			.\label{eq:def:proper-convex}
		\end{equation}
		Moreover, $\psi$ is said to be strictly convex if for $x\neq z,$ inequality \eqref{eq:def:proper-convex} is strict.
	\end{definition}
	
	\begin{remark}
		The classical definition of (strict) convexity \cite{PeypouquetJuan2015COiN} is different to the one we use, however it can be easily recovered by the definition of the Bregman distance (e.g. see \cite[Propoisition 3.10]{PeypouquetJuan2015COiN}).
	\end{remark}

	We can now state a classical result regarding existence and uniqueness of minimizers.

	\begin{lemma}\label{lem:minimizers-unique}
		Let $\psi:\real^n\to\real$ be bounded below,  coercive, and strictly convex. Then $\psi$ has a unique global minimizer.
	\end{lemma}		
	\begin{proof}
		Since $\psi$ is convex and real-valued it is continuous  (see \cite[Corollary~10.1.1]{Rockafellar1970ConvAna}).
		Existence follows from the direct method in the calculus of variations \cite{DacorognaBernard2008DMit}. Uniqueness follows by strict convexity.
	\end{proof}

	%%%%%%%%%%%%%%%%%%%%%%%%%%%%%%%%%%%%%%%%%%%%%%%%%%%%%%%%%%
	
	In particular, by the assumptions on $\op$ and $\reg,$ the lower level problem \eqref{eq:ll} admits a unique minimizer,  justifying our choice of notation.
	
	\begin{proposition}[Properties of the lower level cost function]\label{prop:llfun-properties}
		Fix $\alpha\geq 0.$ Let $\op$ be injective and let $\reg$ be convex, bounded below, and continuously differentiable. Then the lower level cost function $\llfun_\alpha$
		is bounded below, coercive, continuous, and  strictly convex. Moreover, it admits a unique global minimizer.
	\end{proposition}
	\begin{proof}	
		The properties of $\llfun_\alpha$ follow from standard  results in convex analysis.
		Existence and uniqueness of a minimizer follows from Lemma~\ref{lem:minimizers-unique}.
	\end{proof}

	%%%%%%%%%%%%%%%%%%%%%%%%%%%%%%%%%%%%%%%%%%%%%%%%%%%%%%%%%%
	We require continuity of the reconstruction map $\alpha\mapsto\xa$ 
	at $\alpha=0$. 
	Although this is a well studied result \cite{Maso1993GammaConvg,DeLosReyesJ.C2016Tsoo}, we include a full proof for completeness.
	
	%%%%%%%%%%%%%%%%%%%%%%%%%%%%%%%%%%%%%%%%%%%%%%%%%%%%%%%%%%
	
	\begin{lemma}[Convergence of reconstructions at the boundary]\label{lem:convg-at-bdnry} Let $\op$ be injective and let $\reg$ be convex, bounded below, and continuously differentiable.
		If $\{\alpha_k\}\subset (0,\infty)$ satisfies $\lim_{k\to\infty} \alpha_k =0$ then  $\lim_{k\to\infty} \xak = \xo$.
	\end{lemma}
	\begin{proof} 
		By the choice of datafit $\datafit(x)= \frac{1}{2}\norm{\op x - \ynoise}^2$ and injectivity of $\op,$ $\xo$ is the unique minimizer of $\datafit.$ 
		By the minimality of $\xo$
		\begin{align*}
			\datafit(\xo) & \leq \datafit(\xak)
			\intertext{and since $\reg$ is bounded below by some $C\in\real$}
			{\datafit(\xak)} & \leq  \datafit(\xak) + \alpha_k (\reg(\xak) - C)\\
			& = \llfun_{\alpha_k} (\xak)  - \alpha_k C.
			\intertext{By the minimality of $\xak,$}
			\llfun_{\alpha_k} (\xak)  - \alpha_k C & \leq \llfun_{\alpha_k} (\xo)  - \alpha_k C
			\\
			&= \datafit(\xo) + \alpha_k (\reg(\xo) -  C).
		\end{align*}
		In particular, we have $$\datafit(\xo)\leq \datafit(\xak)\leq \datafit(\xo) + \alpha_k (\reg(\xo) - C).$$
		Since $\reg(\xo)$ and $C$ are fixed and $\alpha_k\to 0$, we see that $\lim_{k\to\infty}\datafit(\xak) = \datafit(\xo).$ By continuity of the datafit and \tcb{both minimality and} uniqueness of $\xo$, it follows that 
		$\lim_{k\to\infty} \xak =  \xo$ and we are done.
	\end{proof}

	\section{Proof of the main results}\label{sec:main-result-proof}
	%%%%%%%%%%%%%%%%%%%%%%%%%%%%%%%%%%%%%%%%%%%%%%%%%%%%%%%%%%
	We  aim to characterize when $\alpha=0$ is not a local minimum of the upper level cost function $\ulfun$ by studying local behaviour of $\ulfun$ around $0.$ There are two main challenges towards this: firstly, the reconstruction map $\alpha\mapsto\xa$ is in general non-differentiable \cite{JinBangti2012Srfp} and consequently, without stronger assumptions on the choice of regularizer $\reg$ \cite{Tappen2009LOMEiC,SherryFerdia2020}, the upper level cost function $\ulfun$ is non-differentiable; secondly, the value we are interested in is on the domain boundary of $\ulfun.$ To this end, we work with a generalization of the derivative known as Dini derivatives \cite{ AnsariQamrulHasan.2014Gcnv, GiorgiG,Kannan1996}. More precisely, we consider the upper right Dini derivative.
	
	\begin{definition}[Upper right Dini derivative]
		Let $\mathcal J $ be any real-valued function defined on $[0,\infty)$ and let $\tilde\alpha\geq 0$. We define the upper right Dini derivative of $\mathcal J $ evaluated at $\tilde \alpha$ as
		\begin{equation*}
			\dini \ulfun (\tilde \alpha): = \limsup_{\alpha\to \tilde\alpha_+} \frac{\ulfun(\alpha) - \ulfun(\tilde\alpha)}{\alpha-\tilde\alpha},
		\end{equation*}
		where $\alpha\to\tilde\alpha_+$ denotes the right-hand limit.
	\end{definition}

	We remark that we allow infinite limits in the above definition and so the upper right Dini derivative is always well defined. Dini derivatives follow some standard calculus rules and generalizations of the mean value theorem and Rolle's theorem can be stated~\cite{ AnsariQamrulHasan.2014Gcnv, GiorgiG,Kannan1996}. 
	
	For clarity, we state precisely what we mean for a point to be a local minimum.
	
	\begin{definition}[Local minimum]
		Let $\mathcal J :[0,\infty)\to \real$ be any  function. We say that $\alpha^\star\geq 0$ is a local minimum of $\mathcal J$ if  there exists $\delta>0$ such that
		\begin{equation*}
			\mathcal J(\alpha^\star) \leq \mathcal J (\alpha)  \quad \mbox{for all }  \alpha\in [\alpha^\star - \delta , \alpha^\star + \delta ]\cap [0,\infty).
		\end{equation*}
	\end{definition}

	We now use Dini derivatives to determine behaviour of $\ulfun$ at the domain boundary.
	\begin{lemma}\label{lem:justify-deriv}
		Let $\ulfun:[0,\infty)\to \real$ be any function. If the upper right Dini derivative at $0$ satisfies $\dini\ulfun(0)<0$
		then  $0$ is not a local minimum of $\ulfun$.

	\end{lemma}
	\begin{proof}
		Assume $\dini\ulfun(0)<0.$ We then have existence of $\delta > 0$ such that
		\begin{eqnarray*}
			\frac{\ulfun(\alpha) - \ulfun(0)}{\alpha} < 0 
		\end{eqnarray*}
		for all $\alpha\in(0,\delta)$. It immediately follows that $\ulfun(\alpha) < \ulfun(0)$ in $(0,\delta)$ and so $\ulfun$ is locally strictly decreasing at the domain boundary. In particular, $0$ is not a local minimum of $\ulfun$.
	\end{proof}

	%%%%%%%%%%%%%%%%%%%%%%%%%%%%%%%%%%%%%%%%%%%%%%%%%%%%%%%%%%
	\begin{remark}
		The condition in Lemma~\ref{lem:justify-deriv} is sufficient but not necessary for $0$ to not be a minimizer of $\ulfun.$ Indeed, we will see in Section~\ref{subsec:numerics} that the upper level cost function~\eqref{eq:ul} is non-convex and in particular $0$ may actually be a local minimum, and yet  global minima of \eqref{eq:ul} are achieved at strictly positive parameter values.
	\end{remark}

	We now prove the result of Theorem~\ref{thm:expec} in Section~\ref{subsec:thm:expec} and, using the result of Theorem~\ref{thm:expec}, prove Theorem~\ref{thm:positivity-deblur} in Section~\ref{subsec:forward-op-extension}.
	%%%%%%%%%%%%%%%%%%%%%%%%%%%%%%%%%%%%%%%%%%%%%%%%%%%%%%%%%%
	\subsection{Proof of Theorem~\ref{thm:expec}}\label{subsec:thm:expec}		
	
	Motivated by Lemma~\ref{lem:justify-deriv}, we aim to rewrite $\ulfun(\alpha) - \ulfun(0)$ in a more desirable form, such that the division by $\alpha$ in the Dini derivative will be easier to handle.
	While the bilevel learning problem \eqref{eq:bilevel} involves an expectation, we will find it useful to work with the quantity that we are taking the expectation of and to this end define
	$$
	\tilde \ulfun(\alpha) := \frac{1}{2}\norm{\xa- \xtrue}^2,
	$$
	and so clearly $\ulfun(\alpha) = \E[\tilde\ulfun(\alpha)].$ 
	We focus on calculating $\dini{\tilde\ulfun}(0)$ and build on those results to prove Theorem~\ref{thm:expec}.
	Now, we would ideally like to rewrite $\tilde\ulfun(\alpha)-\tilde\ulfun(0)$ of the form $\alpha h(\alpha)$ for some $h$ such that $\dini{\tilde\ulfun}(0) = \limsup_{\alpha\to 0_+} h(\alpha)<0.$
	While such a form will not be fully achieved, we will determine a form such that the terms that do not have a factor of $\alpha$ will still vanish in the relevant limit.
	More precisely, we want to make a connection between the upper right Dini derivative and regularizer evaluations, and will achieve this using Bregman distances and optimality conditions of $\xa.$ In particular, we will use the following property of the gradient.
	
	%%%%%%%%%%%%%%%%%%%%%%%%%%%%%%%%%%%%%%%%%%%%%%%
	\begin{proposition} \label{prop:reg-grad}
		Let $\reg$ be convex, bounded below, and continuously differentiable and let $\xo$ be any least squares solution.
		Then
		\begin{equation}
			\alpha\nabla\reg(\xa) =  \op^T\op \xo - \op^T\op \xa \label{item-grad}.
		\end{equation}			
	\end{proposition}
	\begin{proof}
		By the choice of the data fidelity and differentiability of $\reg,$ we have that
		\begin{equation}0 = \nabla \llfun_\alpha(\xa) = \op^T\op \xa - \op^T \ynoise + \alpha\nabla\reg(\xa).
			\label{eq:item-grad-tmp}
		\end{equation}
		Since $\xo$ is a least squares solution, it satisfies the normal equations
		\begin{equation}
			\op^T\op \xo - \op^T \ynoise=0 \label{eq:item-grad-tmp2}.
		\end{equation}
		Combining \eqref{eq:item-grad-tmp} and \eqref{eq:item-grad-tmp2} yields \eqref{item-grad}.
	\end{proof}					
	
	%%%%%%%%%%%%%%%%%%%%%%%%%%%%%%%%%%%%%%%%%%%%%%%

	We now state the form of $\tilde \ulfun(\alpha)-\tilde \ulfun(0)$ that will be utilized.		
	
	\begin{proposition}\label{prop:upper-bound-of-dini}
		Let $\op$ be injective and let $\reg$ be convex, bounded below, and continuously differentiable.
		Then
		\begin{align*}
			\tilde\ulfun(\alpha) - \tilde\ulfun(0)
			&=
			\alpha\langle \nabla \reg(\xa) , B\xtrue - B\xa\rangle - \frac{1}{2} \norm{\xo - \xa}^2.
		\end{align*}
	\end{proposition}
	\begin{proof}
		
		By the definition of $\tilde\ulfun$ we have
		\begin{eqnarray*}
			\tilde\ulfun(\alpha)
			=  \frac{1}{2} \norm{\xa-\xtrue}^2 
			=  
			\frac{1}{2} \norm{\xa}^2 
			-
			\langle \xa ,\xtrue\rangle 
			+
			\frac{1}{2} \norm{\xtrue}^2 
		\end{eqnarray*}
		and in particular
		\begin{align*}
			\tilde\ulfun(\alpha) - \tilde\ulfun(0)
			&= \frac{1}{2} \norm{\xa-\xtrue}^2  - \frac{1}{2} \norm{\xo-\xtrue}^2
			\\
			&= 
			\frac{1}{2} \norm{\xa}^2 
			-
			\frac{1}{2} \norm{\xo}^2
			+
			\langle \xo- \xa ,\xtrue\rangle
			\\
			&= 
			\frac{1}{2} \norm{\xa}^2 
			-
			\frac{1}{2} \norm{\xo}^2
			+
			\langle \xo - \xa ,\xtrue - \xa\rangle + \langle \xo,\xa\rangle - \langle \xa,\xa\rangle
			\\
			&= \langle \xo - \xa, \xtrue - \xa\rangle
			-\frac{1}{2}\norm{\xo}^2 
			+
			\langle \xo , \xa\rangle 
			- \frac{1}{2}\norm{\xa}^2
			\\
			&= \langle\xo - \xa, \xtrue - \xa\rangle - \frac{1}{2}\norm{\xo - \xa}^2.
		\end{align*}
		It remains to show that
		\begin{equation*}
			\langle\xo - \xa, \xtrue - \xa\rangle = 
			\alpha
			\langle \nabla \reg(\xa) , B\xtrue - B\xa\rangle.
		\end{equation*}
		Recall that, by Proposition~\ref{prop:reg-grad}, $ \alpha\nabla\reg(\xa)$ involves $\op^T\op$ which we can freely introduce since it is invertible by the injectivity of $\op.$
		Indeed,
		\begin{equation*}
			\langle\xo - \xa, \xtrue - \xa\rangle 
			= 
			\langle B(\op^T\op)(\xo - \xa), \xtrue - \xa\rangle.
		\end{equation*}
		By the symmetry of $B= (\op^T\op)^{-1}$
		\begin{equation*}
			\langle B(\op^T\op)(\xo - \xa), \xtrue - \xa\rangle = 
			\langle \op^T\op(\xo - \xa),  B (\xtrue - \xa)\rangle.
		\end{equation*}
		Finally, by Proposition~\ref{prop:reg-grad}
		\begin{equation*}
			\langle \op^T\op(\xo - \xa),  B (\xtrue - \xa)\rangle   = \alpha
			\langle \nabla \reg(\xa) , B\xtrue - B\xa\rangle
		\end{equation*}
		and we are done.
	\end{proof}
	
	%%%%%%%%%%%%%%%%%%%%%%%%%%%%%%
	
	%%%%%%%%%%%%%%%%%%%%%%%%%%%%%%%%%%%%%%%%%%%%%%%
	In calculating the upper right Dini derivative of $\tilde\ulfun$ at 0, we see from Proposition~\ref{prop:upper-bound-of-dini} that the quantity
	\begin{equation}
		\liminf_{\alpha\to 0_+} \frac{1}{\alpha} \norm{\xo - \xa}^2 \label{eq:term-in-dini-explode-not}
	\end{equation}
	will be encountered.  As hinted earlier, we will show that \eqref{eq:term-in-dini-explode-not} vanishes. Before we do this, we require two intermediate results.

	%%%%%%%%%%%%%%%%%%%%%%%%%%%%%%%%%%%%%%%%%%%%%%%%%%%%%%%%%%%%%%%%%%%
	\begin{proposition}\label{prop:regdif-lb}
		Let $\reg$ be convex, bounded below, and continuously differentiable and let $\xo$ be any least squares solution. Then
		\begin{equation*}
			\alpha (\reg(\xo) - \reg(\xa)) \geq 	\norm{\op(\xo - \xa)}^2.
		\end{equation*}
	\end{proposition}
	\begin{proof}
		By the convexity of $\reg$ and definition of $\breg_\reg,$
		\begin{align*}
			\alpha(\reg(x)-\reg(\xa))
			&\geq 
			\langle \alpha \nabla\reg(\xa)  , x- \xa\rangle.
		\end{align*}
		By Proposition~\ref{prop:reg-grad},
		\begin{equation*}
			\langle \alpha \nabla\reg(\xa)  , x- \xa\rangle =
			\langle \op^T\op \xo - \op^T\op\xa , x- \xa\rangle
		\end{equation*}
		so in taking $x=\xo$ we see that
		\begin{equation*}
			\alpha (\reg(\xo) - \reg(\xa)) \geq 	\norm{\op(\xo - \xa)}^2
		\end{equation*}
		and we are done.
	\end{proof}
	
	%%%%%%%%%%%%%%%%%%%%%%%%%%%%%%%%%%%%%%%%%%%%%%%%%%%%%%%%%%%%%%%%%%%
	%%%%%%%%%%%%%%%%%

	\begin{lemma}\label{lem:justify-ubs}
		Let $f,g:[0,\infty)\to \mathbb{R}$ be functions such that $f(\alpha)\leq g(\alpha)$ for all $\alpha \geq 0.$ Then
		$$
		\limsup_{\alpha\to 0_+} \frac{f(\alpha)}{\alpha} \leq \limsup_{\alpha\to 0_+} \frac{g(\alpha)}{\alpha}.
		$$
	\end{lemma}
	The proof of Lemma~\ref{lem:justify-ubs} is omitted as it follows from the definition of the $\limsup$ and standard arguments in analysis. We now show that \eqref{eq:term-in-dini-explode-not} vanishes.

	\begin{proposition}\label{prop:limsup-results} Let $\op$ be injective and let $\reg$ be convex, bounded below, and continuously differentiable. Then
		\begin{equation*}
			\lim_{\alpha\to 0_+} \frac{1}{\alpha} \norm{\xo - \xa}^2 = 0.
		\end{equation*}
		
	\end{proposition}
	\begin{proof}
		By the non-negativity of $\norm{\xo-\xa}^2/\alpha,$ we immediately have
		\begin{equation*}
			0
			\leq \limsup_{\alpha\to0_+} \frac{1}{\alpha}\norm{\xo-\xa}^2
		\end{equation*} 
		and so the result will follow if we can show that
		\begin{equation}
			\limsup_{\alpha\to0_+} \frac{1}{\alpha}\norm{\xo - \xa}^2 \leq 0.\label{eq:limsup-goal}
		\end{equation}
		Since $\op$ is injective, we have that
		\begin{equation*}
			\norm{\xo - \xa}^2 \leq \frac{1}{\sigma_{\min}^2} \norm{\op(\xo - \xa)}^2,
		\end{equation*}
		where $\sigma_{\min}>0$ is the smallest singular value of $\op.$
		Thus it suffices to show that 
		\begin{equation}
			\limsup_{\alpha\to 0_+} \frac{1}{\alpha} \norm{\op(\xo - \xa)}^2 \leq 0,\label{eq:small-result-tmp}
		\end{equation}
		as then the result will follow by Lemma~\ref{lem:justify-ubs}.
		By Proposition \ref{prop:regdif-lb} and Lemma~\ref{lem:justify-ubs}
		\begin{eqnarray*}
			\limsup_{\alpha \to 0_+} \frac{1}{\alpha} \norm{\op(\xo - \xa)}^2 
			\leq  
			\limsup_{\alpha \to 0_+}\left(\reg(\xo) - \reg(\xa)\right) = 0,
		\end{eqnarray*}
		where the equality follows by the continuity of $\reg$ and Lemma~\ref{lem:convg-at-bdnry}. 
		In particular we have shown \eqref{eq:small-result-tmp} and by Lemma~\ref{lem:justify-ubs} we have \eqref{eq:limsup-goal} and we are done.
	\end{proof}

	%%%%%%%%%%%%%%%%%%%%%%%%%%%%%%%%%%%%%%%%%%%%%%%%%%%%%%%%%%%%%%%
	%%%%%%%%%%%%%%%%%%%%%%%%%%%%%%%%%%%%%%%%%%%%%%%%%%%%%%%%%%%%%%%
	
	We are now ready to prove a crucial result for the proof of Theorem~\ref{thm:expec}.
	In particular, we show that the upper right Dini derivative of $\tilde\ulfun$ at $0$ can be given exactly in terms of  Bregman distances and evaluations of the regularizer.
	
	\begin{lemma}\label{lem:dini-is-negtve}
		Let $\op$ be injective and let $\reg$ be convex, bounded below, and continuously differentiable.  Then
		\begin{equation*} 
			\dini{\tilde\ulfun}(0) = \reg(B \xtrue) - \breg_\reg (B \xtrue,\xo) - 
			\reg(B \xo) + \breg_\reg (B \xo,\xo).
		\end{equation*}
	\end{lemma}
	\begin{proof}
		By Proposition~\ref{prop:upper-bound-of-dini},
		\begin{align}
			\dini{\tilde\ulfun}(0)
			&=\nonumber
			\limsup_{\alpha\to 0_+} \frac{\tilde \ulfun(\alpha) - \tilde\ulfun(0)}{\alpha}
			\\
			&= 
			\limsup_{\alpha\to 0_+} \left( \langle \nabla\reg(\xa) , B\xtrue - B\xa\rangle
			- \frac{1}{2\alpha} \norm{\xo - \xa}^2\right).\label{eq:lemtmp}
		\end{align}
		The last term in \eqref{eq:lemtmp} will vanish in the limit by Proposition~\ref{prop:limsup-results}. 
		By the continuity of $\nabla\reg$ and Lemma~\ref{lem:convg-at-bdnry},
		\begin{align*}
			\dini{\tilde\ulfun}(0) &= \langle \nabla \reg(\xo) , B\xtrue - B\xo\rangle
			\\
			& = \langle \nabla \reg(\xo) , B\xtrue - \xo + \xo -  B\xo\rangle
			\\
			&= \reg (B\xtrue) - \reg(B\xtrue) + \langle \nabla \reg(\xo) , B\xtrue - \xo\rangle 
			- \reg(B\xo) + \reg(B\xo) 
			\\ &\qquad- \langle \nabla \reg(\xo) , B\xo - \xo\rangle 
			\\
			&= 
			\reg(B\xtrue) - \breg_\reg(B\xtrue ,  \xo)
			- \reg(B\xo) + \breg_\reg (B\xo,\xo)
		\end{align*}
		and we are done.
	\end{proof}

	%%%%%%%%%%%%%%%%%%%%%%%%%%%%%%%%%%%%%%%%%%%%%%%%%%%%%%%%%%
	Before we prove Theorem~\ref{thm:expec},  we require an analogous result of Lemma~\ref{lem:dini-is-negtve} for the expected case where, unlike the pointwise setting, we will only be able to find an upper bound of $\dini\ulfun(0).$

	%%%%%%%%%%%%%%%%%%%%%%%%%%%%%%%%%%%%%%%%%%%%%%%%%%%%%%%%%%%%%%%%%%%%%%%%%%%%%%%%%%%%%%%%%%%%%%%%%%%%%%%%%%%%%%%%%%%%%%%%%%%%

	\begin{lemma}\label{lem:dini-deriv-negtve-expec}			
		Let $\op$ be injective and let $\reg$ be convex, bounded below, and continuously differentiable. If
		\begin{align}
			\E\left [	\reg(B\xtrue) + \breg_\reg(B \xo(\ynoise),\xo(\ynoise))  \right] < \infty\label{eq:assmn-finite-expec}
		\end{align}		
		then  $\dini\ulfun(0)\leq \E \left[\dini{\tilde\ulfun}(0)\right].$
	\end{lemma}
	\begin{proof}
		The main challenge is to justify swapping the expectation and $\limsup$ in 
		\begin{equation*}
			\dini\ulfun(0) = \limsup_{\alpha\to 0_+} \E
			\left[ \frac{\tilde \ulfun(\alpha) - \tilde\ulfun(0)}{\alpha} \right].
		\end{equation*}
		This can be justified (up to inequality) by the Reverse Fatou Lemma (e.g. see   \cite[Corollary~5.3.2 ]{ResnickSidneyI2014APP}). 
		We now show that the conditions of the Reverse Fatou Lemma~are satisfied, which in this setting requires showing that $(\tilde\ulfun(\alpha) - \tilde\ulfun(0))/\alpha \leq Z$ for some random variable $Z$ independent of $\alpha$ satisfying $\E[|Z|]<\infty.$

		By Proposition~\ref{prop:upper-bound-of-dini} and the definition of the Bregman distance (see also the proof of Lemma~\ref{lem:dini-is-negtve}), we have
		\begin{align*}
			\tilde\ulfun(\alpha) - \tilde\ulfun(0)
			&\leq 
			\alpha\langle \nabla\reg(\xa) , B\xtrue - B\xa\rangle\\
			& = \alpha \left(
			\reg(B\xtrue) - \breg_\reg(B\xtrue,\xa)
			-\reg(B\xa) + \breg_\reg(B\xa,\xa)
			\right).
		\end{align*}
		Since $\reg$ is bounded below by some $C\in\real $ and $\breg_\reg$ is non-negative it follows that
		\begin{equation*}
			\frac{\tilde\ulfun(\alpha) - \tilde\ulfun(0)}{\alpha} \leq h(\alpha):= \reg(B\xtrue)  - C + \breg_\reg(B\xa,\xa) .
		\end{equation*}
		Notice that by the assumptions on $\reg,$ for any $\alpha \geq0$ we have that $h(\alpha)\geq 0.$ 
		By Lemma~\ref{lem:convg-at-bdnry} and the continuity of $\reg$ and $\nabla\reg$,
		\begin{equation*}
			\lim_{\alpha\to 0_+} h(\alpha) = h(0)
		\end{equation*}
		and so there exists $\delta>0$ such that for all $\alpha\in [0,\delta],$ 
		\begin{equation*}
			h(\alpha) \leq h(0)+1=: Z.
		\end{equation*}
		We claim that this choice of $Z$ is appropriate for the Reverse Fatou Lemma.   Indeed, since we are only interested in behaviour of $\ulfun$
		(and consequently $\tilde\ulfun$) at $0$ we can, without loss of generality, restrict $\ulfun$ to $[0,\delta].$
		Notice that $Z > 0$ and so $ \E [|Z|]=\E [Z] <\infty$ by assumption~\eqref{eq:assmn-finite-expec}. 
		It then follows from the Reverse Fatou Lemma that 
		\begin{align*}
			\dini\ulfun(0) & \leq \E \left[
			\limsup_{\alpha\to 0_+} 	\frac{\tilde\ulfun(\alpha) - \tilde\ulfun(0)}{\alpha}
			\right] = \E \left[\dini{\tilde\ulfun}(0)\right]
		\end{align*}
		and we are done.
	\end{proof}				
	
	%%%%%%%%%%%%%%%%%%%%%%%%%%%%%%%%%%%%%%%%%%%%%%%%%%%%%%%%%%%%%%%%%%%%%%%%%%%%%%%%%%%%%%%%%%%%%%%%%%%%%%%%%%%%%%%%%%%%%%%%%%%%
	We can now prove the main result, which for convenience we restate.
	\newtheorem*{thm:main-expec}{Theorem~\ref{thm:expec}}
	\begin{thm:main-expec}[Positivity of the bilevel learning problem solution]
		Let $\op$ be injective and let $\reg$ be convex, bounded below, and continuously differentiable. If 
		\begin{equation}
			\E \left[ \reg(B\xtrue) - \breg_\reg(B\xtrue,\xo(\ynoise))\right] <
			\E \left[	\reg(B\xo) - \breg_\reg(B\xo(\ynoise),\xo(\ynoise))\right] \label{eq:main-assumn-expec}
		\end{equation} and \begin{equation}
			\E\left [	\reg(B\xtrue) + \breg_\reg(B \xo(\ynoise),\xo(\ynoise))  \right] < \infty, \label{eq:main-assumn-extra}
		\end{equation}		
		then $0$ is not a solution to \eqref{eq:bilevel}.
	\end{thm:main-expec}
	\begin{proof}
		By the assumptions we have from Lemma~\ref{lem:dini-deriv-negtve-expec} that $\dini\ulfun(0)\leq \E[\dini{\tilde\ulfun}(0)].$ 
		By assumption~\eqref{eq:main-assumn-expec} and Lemma~\ref{lem:dini-is-negtve}, we see that $\dini\ulfun(0)<0.$ 
		It follows from Lemma~\ref{lem:justify-deriv} that $0$ is not a local minimizer of $\ulfun$ and in particular cannot be a global minimizer.
	\end{proof}
	
	%%%%%%%%%%%%%%%%%%%%%%%%%%%%%%%%%%%%%%%%%%%%%%%%%%%%%%
	We now show that, in the denoising setting where the measurement has been corrupted by additive noise of mean zero, if the regularizer is strictly convex then we are guaranteed that $0$ is not a solution to \eqref{eq:bilevel}.
	\newtheorem*{cory:sym-breg}{Corollary~\ref{cor:sym-breg}}
	\begin{cory:sym-breg}
		Let $\op=I$ and let $\reg$ be convex, bounded below, and continuously differentiable.
		If $\xtrue$ is fixed and $\ynoise = \xtrue + \epsilon$ where $\E_\epsilon[\epsilon]=0$, then \tcb{we have the following.}
		\begin{enumerate}[label=(\roman*)]
			\item \tcb{Condition \eqref{eq:main-assumn-expec} is equivalent to}
			\begin{equation}
				0< \E_{\epsilon}\left[
				\breg_R(\ynoise,\xtrue) + \breg_\reg(\xtrue,\ynoise)\label{eq:sym-breg2}
				\right].
			\end{equation}
			\item \tcb{If $\reg$ is strictly convex, then $0$ is not a solution to \eqref{eq:bilevel} almost surely.}
		\end{enumerate}				
	\end{cory:sym-breg}
	\begin{proof}
		By the definition of the Bregman distance, \eqref{eq:main-assumn-expec} is also given by
		\begin{align}
			0 &< \E_\epsilon\left[\nonumber
			\langle\nabla\reg(\ynoise),\ynoise-\xtrue\rangle
			\right]
			\\
			&= \E_\epsilon\left[\nonumber
			\langle\nabla\reg(\ynoise) - \nabla\reg(\xtrue) + \nabla\reg(\xtrue),\ynoise-\xtrue\rangle
			\right]
			\\
			&=\E_\epsilon\left[ \label{eq:cor-temp-result}
			\langle\nabla\reg(\ynoise) - \nabla\reg(\xtrue),\ynoise-\xtrue\rangle
			+ \langle\nabla\reg(\xtrue),\ynoise-\xtrue\rangle
			\right].
		\end{align}
		In the denoising setting we have that $\ynoise-\xtrue=\epsilon.$ Since $\xtrue$ is fixed and $\E_\epsilon[\epsilon]=0,$ it follows from \eqref{eq:cor-temp-result} that
		\begin{align*}
			0 & < \E_\epsilon\left[
			\langle\nabla\reg(\ynoise) - \nabla\reg(\xtrue),\ynoise-\xtrue\rangle\right]
			+ \langle\nabla\reg(\xtrue),\E_\epsilon[\epsilon]\rangle
			\\&= 
			\E_\epsilon\left[
			\langle\nabla\reg(\ynoise) - \nabla\reg(\xtrue),\ynoise-\xtrue\rangle\right].
		\end{align*}
		By definition of the Bregman distance, we have
		\begin{equation*}
			\E_\epsilon\left[
			\langle\nabla\reg(\ynoise) - \nabla\reg(\xtrue),\ynoise-\xtrue\rangle\right]
			=
			\E\left[
			\breg_R(\ynoise,\xtrue) + \breg_\reg(\xtrue,\ynoise)
			\right]
		\end{equation*}
		and we have shown \eqref{eq:sym-breg2}.
		Since $\epsilon$ is a continuous random variable we have that $\xtrue\neq\ynoise$ almost surely.
		From the definition of strict convexity it immediately follows that $\breg_\reg(y,\xtrue) >0$ and so \eqref{eq:sym-breg2} is  satisfied almost surely.	 Since $\op=I$ and $\xtrue$ is fixed, the other condition of Theorem~\ref{thm:expec} is trivially satisfied and so $0$ is not a minimum of \eqref{eq:bilevel}  by Theorem~\ref{thm:expec}.
	\end{proof}

	%%%%%%%%%%%%%%%%%%%%%%%%%%%%%%%%%%%%%%%%%%%%%%%%%%%%%%

	We now provide a sufficient condition for assumption \eqref{eq:main-assumn-extra} to be satisfied.
	\begin{proposition}
		Let $\op$ be injective and let $\reg$ be convex, bounded below, continuously differentiable, and $\beta$-smooth, in that, for any $x,z \in\real^n$
		\begin{equation*}
			\norm{\nabla\reg(x) - \nabla\reg(z)} \leq \beta \norm{x-z}.
		\end{equation*}
		If
		\begin{equation}
			\E\left[ \reg(B\xtrue)  \right]<\infty\quad\text{and}\quad\E\left[ \|{\xo(\ynoise)}\|^2  \right]<\infty,\label{eq:cory-suffic-assmn}
		\end{equation}
		then \eqref{eq:main-assumn-extra} is satisfied.
	\end{proposition}
	\begin{proof}
		By the definition of the Bregman distance it follows from~\cite[Theorem 5.8]{BeckAmir2017Fmio} that $\beta$-smoothness is equivalent to 
		\begin{equation*}
			\breg_\reg(x,z) \leq \frac{\beta}{2} \norm{x-z}.
		\end{equation*}
		It follows that
		\begin{align*}
			\E \left[\reg(B\xtrue) + \breg_\reg( B\xo , \xo)\right]
			& \leq \E \left[\reg(B\xtrue)\right] +
			\frac{\beta}{2}\E\left[ \norm{ (B - I) \xo}^2 \right]
			\\
			& \leq \E \left[\reg(B\xtrue)\right] +
			\frac{\beta\sigma_{\max}^2}{2}\E \left[ \norm{ \xo}^2 \right]
		\end{align*}
		where $\sigma_{\max}\geq 0$ is the largest singular value of $B-I.$ 
		Thus \eqref{eq:cory-suffic-assmn} is sufficient for \eqref{eq:main-assumn-extra} to be satisfied.
	\end{proof}
	%%%%%%%%%%%%%%%%%%%%%%%%%%%%%%%%%%%%%%%%%%%%%%%%%%%%%%
	%%%%%%%%%%%%%%%%%%%%%%%%%%%%%%%%%%%%%%%%%%%%%%%%%%%%%%
	%%%%%%%%%%%%%%%%%%%%%%%%%%%%%%%%%%%%%%%%%%%%%%%%%%%%%% 
	
	\subsection{Proof of Theorem~\ref{thm:positivity-deblur}}\label{subsec:forward-op-extension}
	We now prove an analogous result of Theorem~\ref{thm:expec} for the predictive risk upper level cost function
	and an invertible forward operator $\op$, wherein the associated bilevel learning problem is
	\begin{subequations}\label{eq:bilevel-pred-risk}
		\begin{equation}
			\min_{\alpha\in\mathcal P} \E\left[ \frac{1}{2}\norm{\op\xa-\op\xtrue}^2\right], \label{eq:ul-deblur}
		\end{equation}
		\begin{equation}
			\text{subject to }	\xa = \arg\min_{x\in\real^n} \left\{\frac{1}{2} \norm{\op x-\ynoise}^2 + \alpha \reg(x)  \right\}  \label{eq:ll-deblur}.
		\end{equation}
	\end{subequations}
	Using Theorem~\ref{thm:expec} and the invertibility of $\op$, positivity of solutions  of the bilevel learning problem \eqref{eq:bilevel-pred-risk} can be established.
	
	\newtheorem*{thm:positivity-deblur}{Theorem~\ref{thm:positivity-deblur}}
	\begin{thm:positivity-deblur}
		Let $\op$ be invertible and let $\reg$ be convex, bounded below, and continuously differentiable. If				\begin{equation}
			\E\left[	\reg( \xtrue) - \breg_\reg(\xtrue,\xo(\ynoise) )\right] <
			\E\left[\reg(\xo(\ynoise))\right]\label{eq:reg-cond-z}\end{equation} and \begin{equation}
			\E\left[\reg( \xtrue)\right] <\infty,\label{eq:reg-cond-z3}
		\end{equation}
		then $0$ is not a solution to \eqref{eq:bilevel-pred-risk}.
	\end{thm:positivity-deblur}
	\begin{proof}
		Using the invertibility of $\op,$ we intend to rephrase the bilevel learning problem \eqref{eq:bilevel-pred-risk} as a denoising problem and apply Theorem~\ref{thm:expec}.
		
		We first remark that since $\op$ is assumed invertible, $\xo = \op^{-1} \ynoise$
		and more generally the solution $\xa$ to \eqref{eq:ll-deblur} satisfies $\op\xa = \za$, where
		\begin{equation}
			\za = \arg\min_{z\in\real^n} \left\{\frac{1}{2} \norm{z-\ynoise}^2 + \alpha \reg(\op^{-1}z)  \right\}. \label{eq:ll-deblur-z}
		\end{equation}
		By the invertibility of $A$ and assumptions on $\reg$ it follows that $\tilde \reg := \reg\circ \op^{-1}$ is also convex, bounded below, and continuously differentiable.
		Furthermore, if we define $\ztrue:= \op\xtrue$ notice that
		\begin{align*}
			\breg_\reg(\xtrue , \xo) 
			&= \reg(\xtrue) - \langle \nabla\reg(\xo) , \xtrue - \xo\rangle
			\\
			&= \reg(\xtrue) - \langle \op^{-T} \nabla \reg(\xo) , \op\xtrue - \op\xo\rangle
			\\
			&= \tilde\reg(\ztrue) - \langle \nabla\tilde\reg(\ynoise) , \ztrue - \ynoise\rangle
			\\
			&= \breg_{\tilde\reg}(\ztrue , \ynoise).
		\end{align*}	
		Thus, assumption \eqref{eq:reg-cond-z} reads
		\begin{equation}\E \left[ 
			\tilde\reg(\ztrue) - \breg_{\tilde\reg} (\ztrue,\ynoise)
			\right] < \E\left[\tilde \reg(\ynoise) 
			\right] \label{eq:reg-cond-z2}
		\end{equation}		
		and assumption	\eqref{eq:reg-cond-z3} reads
		\begin{equation}
			\E\left[\tilde \reg( \ztrue)\right] <\infty.\label{eq:reg-cond-z4}
		\end{equation}
		Using the new notation, the upper level problem \eqref{eq:ul-deblur} reads
		\begin{equation}
			\min_{\alpha\in\mathcal P} \E\left[ \frac{1}{2}\norm{\za-\ztrue}^2\right]. \label{eq:ul-deblur-z}
		\end{equation}
		In particular, we have phrased the bilevel learning problem \eqref{eq:bilevel-pred-risk} as a denoising bilevel problem \eqref{eq:ul-deblur-z} and \eqref{eq:ll-deblur-z} with regularizer $\tilde \reg$. By the properties of $\tilde \reg$ and inequalities \eqref{eq:reg-cond-z2} and \eqref{eq:reg-cond-z4}, it follows from Theorem~\ref{thm:expec} 
		that $0$ is not a solution to  \eqref{eq:bilevel-pred-risk}.
	\end{proof}

	%%%%%%%%%%%%%%%%%%%%%%%%%%%%%%%%%%%%%%%%%%%%%%%%%%%%%%%%%%%%%%%%%%%%%
	%%%%%%%%%%%%%%%%%%%%%%%%%%%%%%%%%%%%%%%%%%%%%%%%%%%%%%%%%%%%%%%%%%%%%
	%%%%%%%%%%%%%%%%%%%%%%%%%%%%%%%%%%%%%%%%%%%%%%%%%%%%%%%%%%%%%%%%%%%%%
	\section{Numerical Experiments}\label{subsec:numerics}
	We now explore the presented theory with some numerical examples. Although in practice the problem is high-dimensional, for a geometric and visual interpretation of the theory, we consider in Section \ref{subsec:numerics:lowscale} the small dimensional case of $n=2$. Relevant high-dimensional problems are provided in Section \ref{subsec:numerics:largescale}. 
	
	In the following, solutions to the lower level problem are computed either via an analytic closed form solution \tcb{(where no numerical solver is necessary)} or numerically using BFGS \cite{NocedalWright2006NumOp} for \tcb{10000} iterations with backtracking linesearch and early stopping if either the gradient norm evaluates to less than $10^{-8}$ \tcb{or a step length smaller than $10^{-14}$ is considered by backtracking linesearch. 
		We find that early stopping is always achieved.
		We utilize both the SciPy~\cite{Virtanen2020SciPy} and ODL~\cite{Adler2017ODL} Python libraries in our experiments, and our code is available at  \href{https://github.com/s-j-scott/Optimal-Reg-Params-Bilevel}{https://github.com/s-j-scott/Optimal-Reg-Params-Bilevel}.
	}
	
	\subsection{Low-dimensional problems}\label{subsec:numerics:lowscale}

	We explore how well the results of  Theorem~\ref{thm:expec} and Theorem~\ref{thm:positivity-deblur} characterize positivity in the pointwise setting, that is, we consider Corollary~\ref{cory:pointwise} and Corollary~\ref{cory:pointwise-deblur} respectively. 				
	In the following, we consider the area $\Omega:=[-1.6,1.6]\times [-1.6,1.6]\subset \real^2,$ discretised into a $150\times150$ grid.
	Since our results involve $\xo$ and $\xtrue,$ we interpret $\Omega$ as the reconstruction space, rather than the measurement space. 
	
	We fix the ground truth $\xtrue = [1,0.5]^T,$ which will be indicated by a yellow star in the upcoming plots. 
	Considering each point in the grid $\Omega$ as a candidate  $\xo,$  we compute the boundary for when the relevant inequality of  Corollary~\ref{cory:pointwise} or Corollary~\ref{cory:pointwise-deblur} becomes satisfied. If we are in the case $\op=I,$ we may also compute the boundary for when \eqref{eq:heuristic-reg-cond} becomes satisfied. 
	Since the lower level problem requires a measurement $\ynoise$, in order to have a well defined mapping between the $\xo$ and $\ynoise,$ we restrict ourselves to invertible $\op$ in this section.
	We  approximate the solution to the bilevel learning problem with data $(\xtrue , \op\xo)$ by considering parameters
	$$  [\alpha_1=0,\;\alpha_2=10^{-8},\cdots,\;\alpha_{99}=10^3,\;\alpha_{100}=10^7],
	$$
	where $[\log_{10}(\alpha_2),\;\cdots,\log_{10}(\alpha_{99})]$ is a linear discretisation of 98 points between -8 and 3, and select the parameter that achieves the smallest upper level cost. With this, we compute the boundary  in $\Omega$ between the regions where  the numerical bilevel solution is zero and strictly positive.
	A summary of the boundaries and their represented colours and legend names is provided in Table~\ref{tab:numerics-colour-legend}.
	\begin{table}
		\centering
		\caption{\label{tab:numerics-colour-legend}Key for the boundary colours used in numerical examples. The shorthand labels used in the following legends are also indicated.
			Since the results of Corollary~\ref{cory:pointwise} and Corollary~\ref{cory:pointwise-deblur} refer to different bilevel problems, \eqref{eq:bilevel} and \eqref{eq:bilevel-pred-risk-og} respectively, we can re-use the label and colour for both boundaries.}
		\begin{tabular}{lllc}\toprule
			Condition satisfied outside the boundary&Colour&Label&Reference\\\midrule
			$\reg(\xtrue) < \reg(\ynoise)$    &Red & Old&\cite{DavoliElisa2023SCiN,DeLosReyesJ.C2016Tsoo}   
			\\
			$ 
			\reg(B \xtrue) - \breg_\reg (B \xtrue,\xo) < 
			\reg(B \xo) - \breg_\reg (B \xo,\xo)$& Blue&New&Corollary~\ref{cory:pointwise}
			\\
			$ 
			\reg( \xtrue) - \breg_\reg ( \xtrue,\xo) < 
			\reg( \xo) $& Blue&New&Corollary~\ref{cory:pointwise-deblur}
			\\
			Numerical solution to the bilevel problem is not $0$ &Black&Numerical&--  \\
			\bottomrule
		\end{tabular}
	\end{table}

	We consider two forward operators, namely,
	\begin{align*}
		\op_1 = \begin{bmatrix}
			1 & 0 \\ 0& 1
		\end{bmatrix}\quad \text{and}\quad	
		\op_2 = \begin{bmatrix}
			0.7274 & 0.2726 \\ 0.2726 & 0.7274
		\end{bmatrix}.
	\end{align*}
	
	For each forward operator, we consider four different regularizers and see how the boundaries, detailed in the above and summarized in Table~\ref{tab:numerics-colour-legend}, change.
	In particular, we consider a general form of  Tikhonov and  the Huber norm,  
	\begin{equation*}
		\reg(x) = \frac{1}{2}\norm{ L x}^2 \quad \text{and}\quad	\reg(x) = \sum_{i=1}^p \mathrm{hub}_\gamma ([L x]_i)
	\end{equation*}
	respectively, where $L\in\real^{p\times n},$ not necessarily full rank, and 
	\begin{equation*}
		\mathrm{hub}_\gamma(t) = \left\{ \begin{array}{ll}
			\abs{t} - \frac{\gamma}{2} & \text{if } \abs{t}\geq \gamma
			\\
			\frac{1}{2\gamma} t^2 & \text{if } \abs{t} < \gamma.
		\end{array}\right.
	\end{equation*}
	Regarding the choice of $L,$ we will consider both  $L=I\in\real^{2\times 2},$ which will yield standard Tikhonov and Huber norm respectively, and also  $L=[1\;-1]\in\real^{1\times 2}$ which can be interpreted as the discretisation of the first-order finite difference operator for $n=2$ \cite{HansenPerChristian2006Di:m}. 
	For this latter choice of $L,$ we  refer to the regularizers as an $n=2$  analogue of the  $H^1$ seminorm  and   Huber TV respectively.
	
	%%%%%%%%%%%%%%%%%%%%%%%%%%%%%%%%%%%%%%%%%%%%%%%%%%%%%%%%%%%%%%%%%
	\subsubsection{On the characterization of positivity}
	We are interested in how well Corollary~\ref{cory:pointwise} characterizes positivity of solutions of the bilevel learning problem \eqref{eq:bilevel}.
	Using the approach outlined above, we can determine the area of the region where the numerical solution to \eqref{eq:bilevel} is 0, and also the area where the condition of Corollary~\ref{cory:pointwise} is violated.
	Should Corollary~\ref{cory:pointwise} perfectly characterize positivity, we would expect both areas to coincide.
	We compute the ratio between these areas for the different $\op$ and $\reg$ mentioned above, and display the results in Table~\ref{tab:numerics-area-increase}. 
	In the denoising setting, we also compute the area where \eqref{eq:heuristic-reg-cond} is violated, to see how the new condition compares.
	In Figure~\ref{fig:numeric-tikh-denoise} we see that, for Tikhonov denoising, Corollary~\ref{cory:pointwise} perfectly characterizes positivity, as we would expect following Example \ref{ex:perfect-characterization}. 
	From Table~\ref{tab:numerics-area-increase}, Corollary~\ref{cory:pointwise} characterizes the positivity of \eqref{eq:bilevel} well for the considered problems, with many area ratios being around 1. 
	Furthermore, we see in Figure~\ref{fig:numer-small-denoise} and Figure~\ref{fig:numer-small-deblur} that some instances where the ratio is close to but not exactly 1 is down to numerical error.
	For the denoising setting, we see in Figure~\ref{fig:numer-small-denoise} that \eqref{eq:heuristic-reg-cond} overestimates the region where $0$ is a solution to \eqref{eq:bilevel} by a factor of 2 to 4.
	Compared to condition \eqref{eq:heuristic-reg-cond}, Corollary~\ref{cory:pointwise} yields a better characterization of positivity, particularly for points far away from $\xtrue$ - as demonstrated in Figure~\ref{fig:numeric-tikh-denoise} and Figure~\ref{fig:numeric-hub-denoise}.

	\begin{table}[!ht]
		\centering
		\caption{\label{tab:numerics-area-increase}Ratio between the area where $0$ is the optimal parameter and area in the reconstruction space where the (old or new) theory condition is violated. 
			Values close to $1$ mean the condition is close to fully characterizing positivity of \eqref{eq:bilevel}.
			Since \eqref{eq:heuristic-reg-cond} is only valid for $\op=I,$ we cannot compare for the $\op\neq I$ case. 
			As we only consider points in $\Omega,$ if the area where a condition is violated extends beyond $\Omega,$ we indicate the case with an asterisk beside the provided number. All numbers are given to 3 decimal points.
		}
		\begin{tabular}{ccllll}
			\toprule 
			\multirow{2}{*}{Problem} &  \multirow{2}{*}{Condition violated} & \multicolumn{4}{c}{Regularizer} \\
			\cmidrule{3-6}
			&& Tikhonov & $H^1$ seminorm & Huber & Huber TV \\\midrule
			Denoising    &New& 1    & 1.016*& 1.172& 0.985*\\
			Denoising    &Old& 4.008& 2.047*& 4.207& 1.985*\\
			Non-denoising&New& 1.028& 0.981*& 1.181& 0.983*\\
			Non-denoising&Old& ---		& 	---		& ---		& 	---		\\
			% Denoising		& New	& 1			& 	1.069*	& 1.171		& 1.129*	\\
			% Denoising		& Old	& 3.979		& 	2.071*	& 4.214		& 2.182*			\\
			% Non-denoising	& New	& 1.028		& 	1.020* 	&  1.143	& 1.015*			\\
			%  Non-denoising	& Old	& ---		& 	---		& ---		& 	---		\\
			\bottomrule
		\end{tabular}
	\end{table}
	%%%%%%%%%%%%%%%%%%%%%%%%%%%%%%%%%%%%%%%%%%%%%%%%%%%%%%

	\begin{figure}
		\centering
		\begin{subfigure}{0.47\textwidth}
			\includegraphics[width=\textwidth]{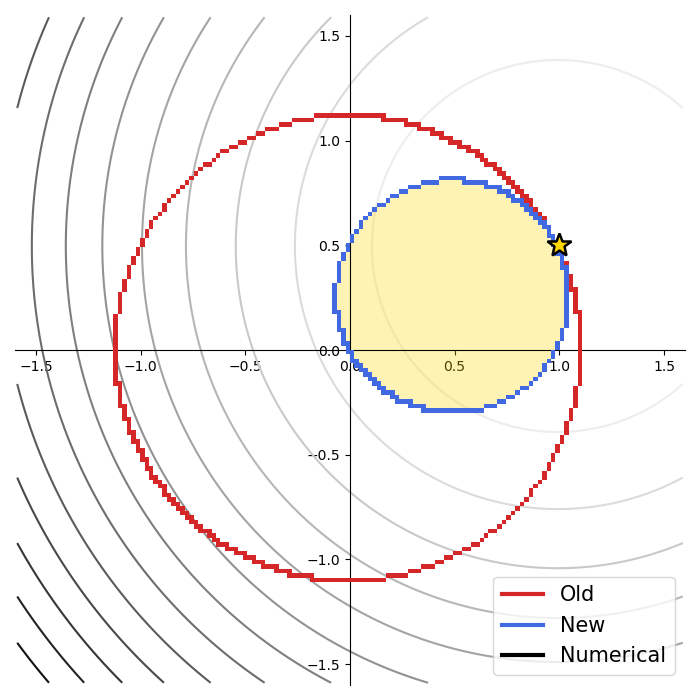}
			\caption{ Tikhonov regularization}\label{fig:numeric-tikh-denoise}
		\end{subfigure}
		\begin{subfigure}{0.47\textwidth}
			\includegraphics[width=\textwidth]{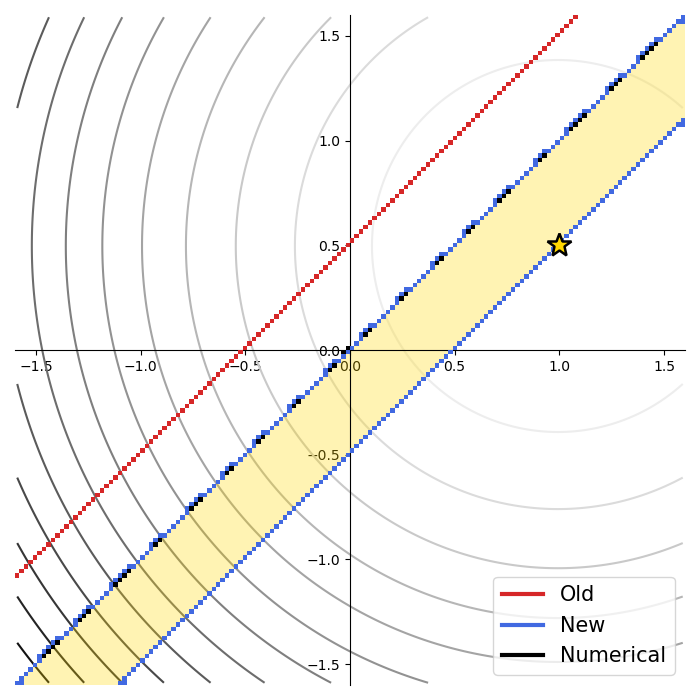}
			\caption{$H^1$ seminorm regularization}\label{fig:numeric-l2grad-denoise}
		\end{subfigure}
		\begin{subfigure}{0.47\textwidth}
			\includegraphics[width=\textwidth]{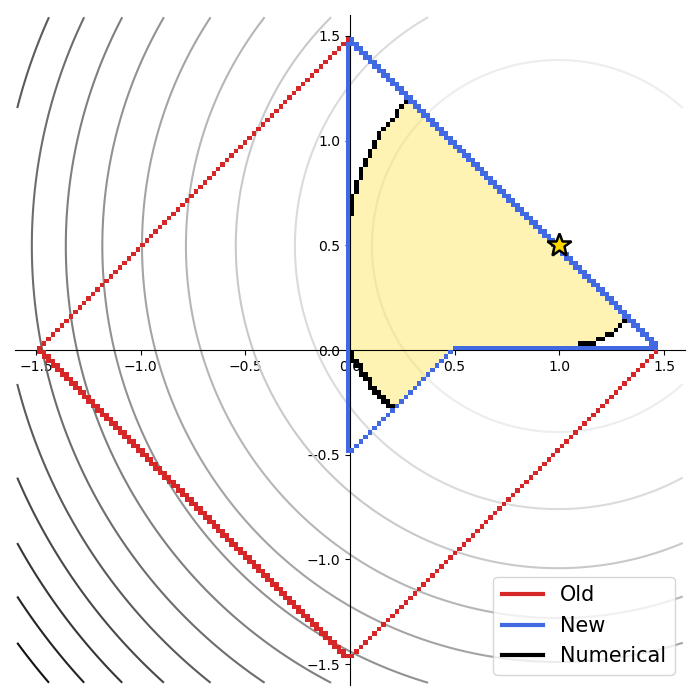}
			\caption{Huber regularization}\label{fig:numeric-hub-denoise}
		\end{subfigure}		
		\begin{subfigure}{0.47\textwidth}
			\includegraphics[width=\textwidth]{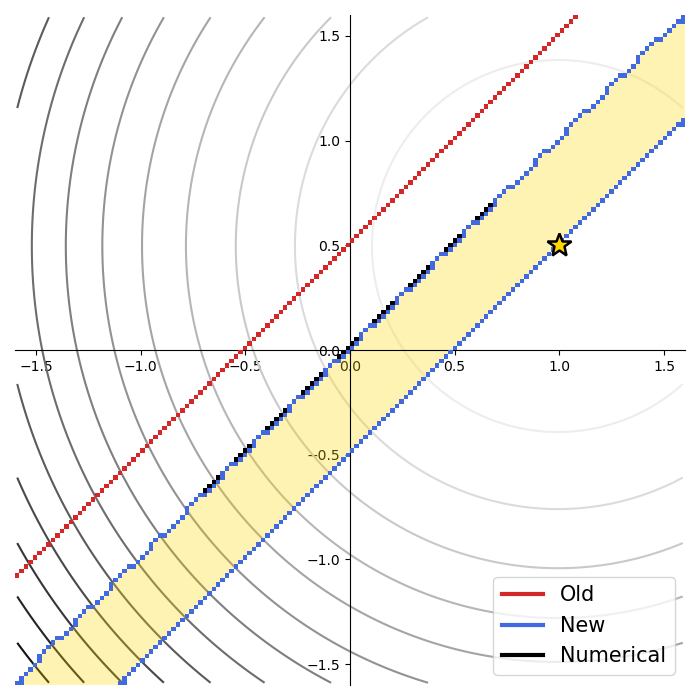}
			\caption{Huber TV regularization}
		\end{subfigure}
		
		\caption{Visualisation of Corollary~\ref{cory:pointwise}. Plots of the reconstruction space $\Omega$ for the trivial forward operator ($\op_1$) setting and various choices of regularizer,  with the condition boundaries as detailed in Table~\ref{tab:numerics-colour-legend}.
			The ground truth $\xtrue=[1,0.5]$ is represented by a yellow star, and level sets of the upper level cost function are visible. 
			The region where $0$ is a solution to \eqref{eq:bilevel} is shaded yellow.
			The choice of regularizer is indicated \tcb{in each subcaption.}  }
		\label{fig:numer-small-denoise}
	\end{figure}

	%%%%%%%%%%%%%%%%%%%%%%%%%%%%%%%%%%%%%%%%%%%%%%%%%%%%%%%%%%%%	

	\begin{figure}
		\centering
		\begin{subfigure}{0.47\textwidth}
			\includegraphics[width=\textwidth]{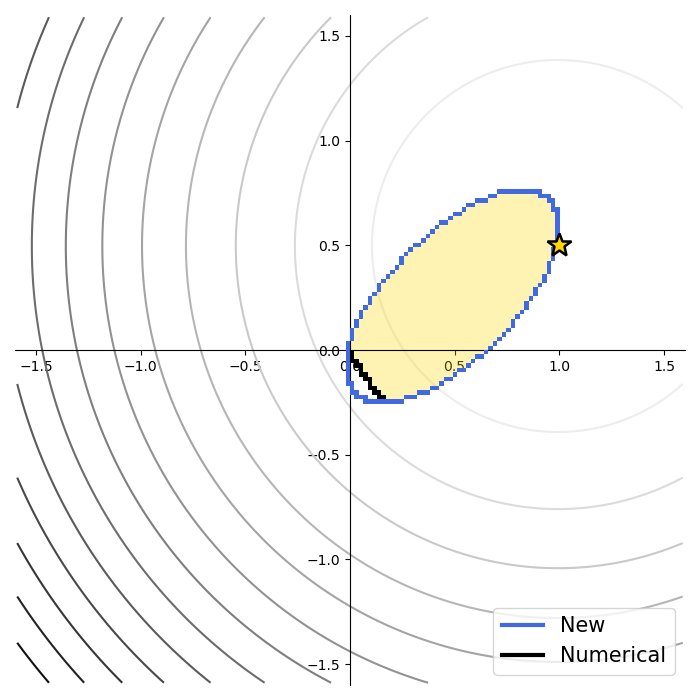}
			\caption{ Tikhonov regularization}
		\end{subfigure}
		\begin{subfigure}{0.47\textwidth}
			\includegraphics[width=\textwidth]{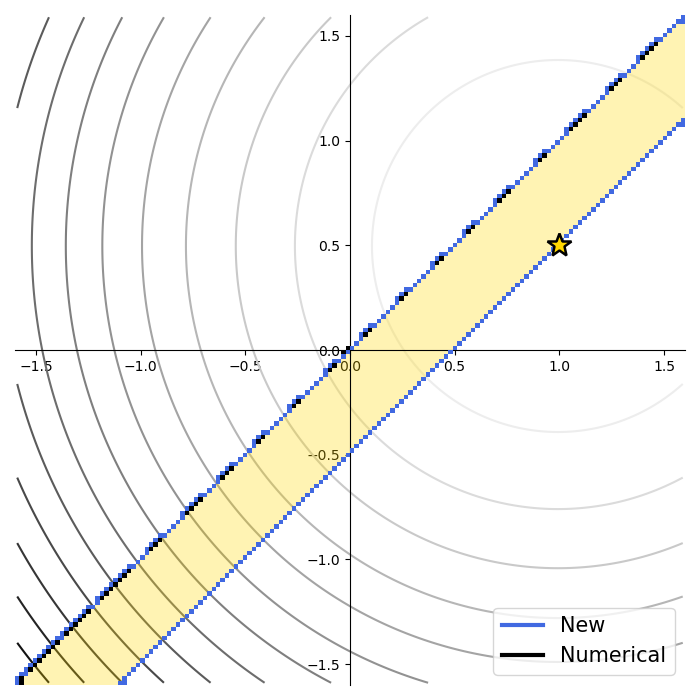}
			\caption{$H^1$ seminorm  regularization}
		\end{subfigure}
		\begin{subfigure}{0.47\textwidth}
			\includegraphics[width=\textwidth]{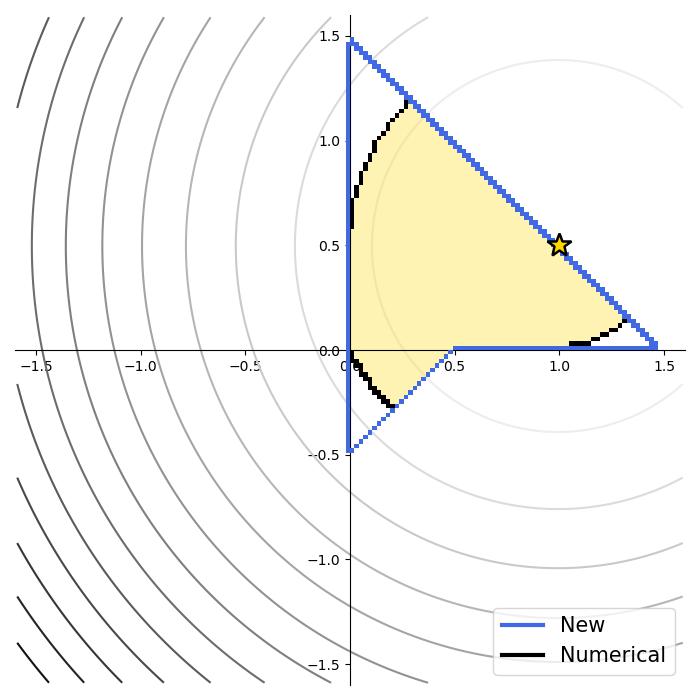}
			\caption{Huber norm regularization}
		\end{subfigure}		
		\begin{subfigure}{0.47\textwidth}
			\includegraphics[width=\textwidth]{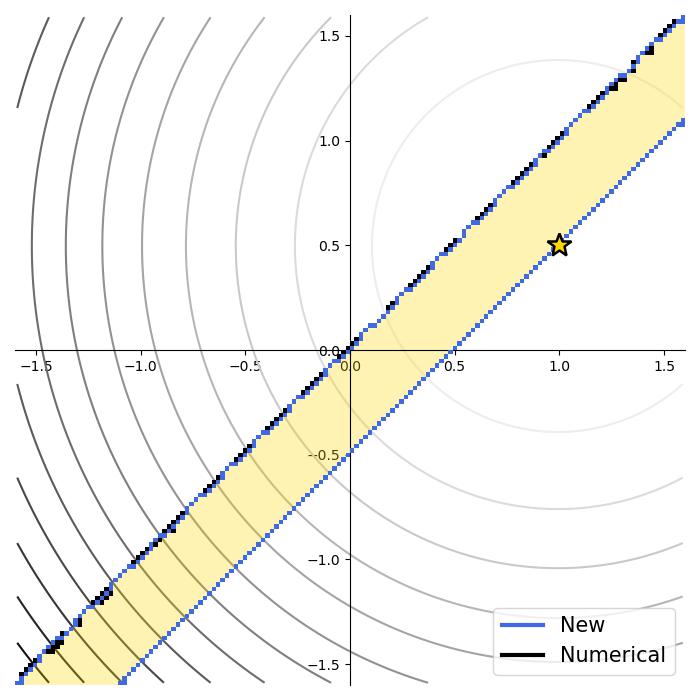}
			\caption{Huber TV  regularization}
		\end{subfigure}
		
		\caption{ Visualisation of Corollary~\ref{cory:pointwise}.
			Plots of the reconstruction space $\Omega$ for the non-trivial forward operator ($\op_2$) setting and various choices of regularizer,  with the condition boundaries as detailed in Table~\ref{tab:numerics-colour-legend}.
			The ground truth $\xtrue=[1,0.5]^T$ is represented by a yellow star, and level sets of the upper level cost function are visible. 
			The region where $0$ is a solution to \eqref{eq:bilevel} is shaded yellow. 
			The choice of regularizer is indicated \tcb{in each subcaption.} }
		\label{fig:numer-small-deblur}
	\end{figure}

	%%%%%%%%%%%%%%%%%%%%%%%%%%%%%%%%%%%%%%%%%%%%%%%%%%%%%%%%%%%%	
	%%%%%%%%%%%%%%%%%%%%%%%%%%%%%%%%%%%%%%%%%%%%%%%%%%%%%%%%%%%%	
	%%%%%%%%%%%%%%%%%%%%%%%%%%%%%%%%%%%%%%%%%%%%%%%%%%%%%%%%%%%%			
	
	\subsubsection{Guaranteed positivity for denoising}			
	
	We now demonstrate the result of Corollary~\ref{cor:sym-breg}, where we are guaranteed that $0$ is not a solution provided that $\op=I,$  $\reg$ is strictly convex, and the noise is additive and of zero mean. We fix ground truth $\xtrue = [1,0]^T$ and generate 1000 noisy realisations by perturbing $\xtrue$ with Gaussian noise of mean $[0,0]^T,$ standard deviation $[0.1,0.1]^T.$
	A plot of the ground truth and noisy realisations is shown in Figure~\ref{fig-cory-positivity-data-centre}. 
	To ensure the regularizer is strictly convex and differentiable, we consider
	\begin{equation*}
		\reg(x) = \frac{\beta}{2} \norm{x}^2 + \sum_{i=1}^n \mathrm{hub}_\gamma (x_i),
	\end{equation*}
	for $\beta=\gamma=0.01.$ 
	For regularization parameters in the linear discretisation of the interval $[0,0.1]$ into 50 points, we plot the associated upper level cost in Figure~\ref{fig-cory-positivity-ul-centre}.
	We see that the optimal parameter is achieved at a strictly positive value.
	We now show that if the assumption on the noise is violated, we are not guaranteed positivity. 
	For the same $\xtrue$ we generate 1000 noisy realisations by perturbing $\xtrue$ with Gaussian noise of mean $[-0.1,0]^T$ and standard deviation $[0.1,0.1]^T.$
	A plot of the ground truth and noisy realisations is shown in Figure~\ref{fig-cory-positivity-data-uncentre}, and the associated upper level cost in Figure~\ref{fig-cory-positivity-ul-uncentre}. 
	We see that $0$ is the optimal parameter in this case 
	which indicates that the assumption of zero mean noise in Corollary~\ref{cor:sym-breg} is tight.

	\begin{figure}
		\centering
		\begin{subfigure}[b]{0.47\textwidth}
			\centering
			\includegraphics[width=\textwidth]{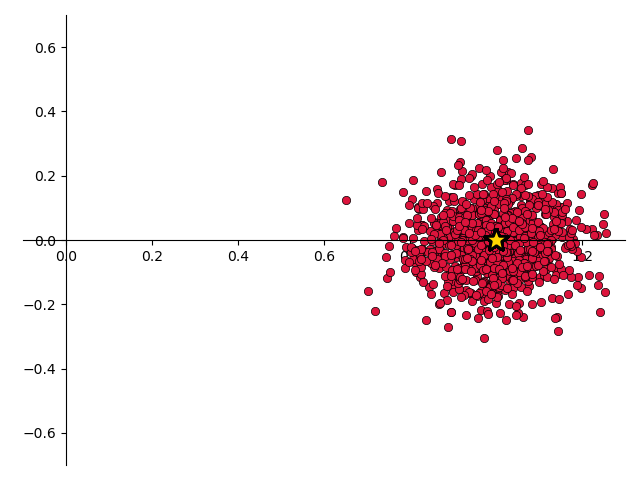}
			\caption{}
			\label{fig-cory-positivity-data-centre}
		\end{subfigure}
		\begin{subfigure}[b]{0.47\textwidth}
			\centering
			\includegraphics[width=\textwidth]{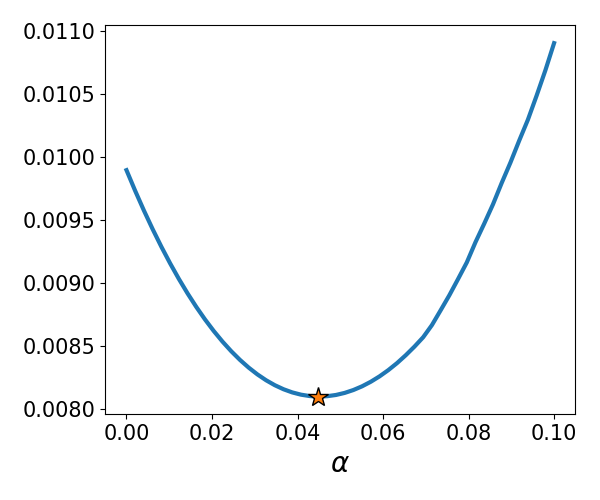}
			\caption{}
			\label{fig-cory-positivity-ul-centre}
		\end{subfigure}
		\begin{subfigure}[b]{0.47\textwidth}
			\centering
			\includegraphics[width=\textwidth]{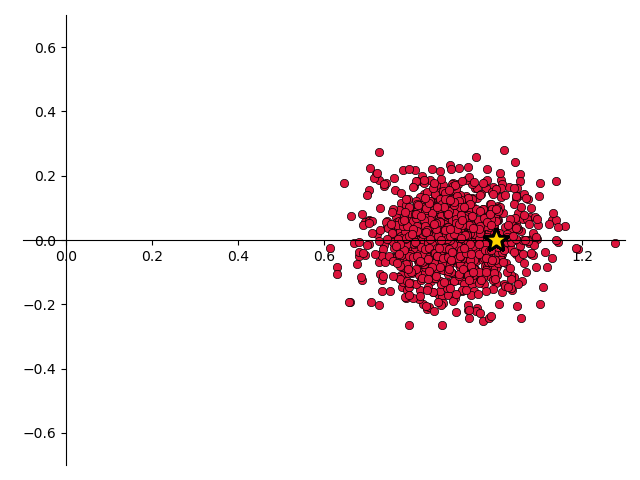}
			\caption{}
			\label{fig-cory-positivity-data-uncentre}
		\end{subfigure}
		\begin{subfigure}[b]{0.47\textwidth}
			\centering
			\includegraphics[width=\textwidth]{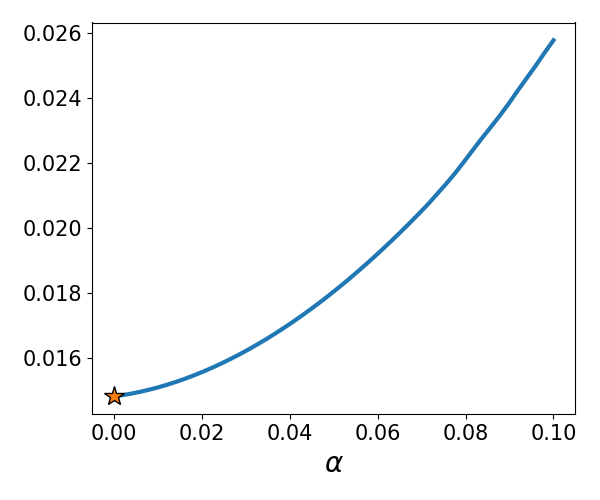}
			\caption{}
			\label{fig-cory-positivity-ul-uncentre}
		\end{subfigure}
		\caption{Visualisation of Corollary~\ref{cor:sym-breg}.
			\tcb{(a) shows the} ground truth $\xtrue = [1,0]^T$ indicated by a yellow star, and 1000 noisy realisations indicated by red dots where the corruption was additive Gaussian noise of mean $[-0.1,0]^T$ standard deviation $[0.1,0.1]^T.$
			\tcb{(b) shows the} squared error upper level cost corresponding to the data in \tcb{(a)}. The optimal regularization parameter is indicated by an orange star  which, since the conditions of Corollary~\ref{cor:sym-breg} are satisfied, is guaranteed to not be  $0.$
			\tcb{(c) shows} a similar to \tcb{plot to} \tcb{(a)}, but the noise has non-zero mean $[-0.1,0]^T,$ and so the result of {Corollary}~\ref{cor:sym-breg} is not applicable.
			\tcb{(d) shows the} squared error upper level cost corresponding to the data in \tcb{(c)}. The optimal regularization parameter is indicated by an orange star. Since the conditions of {Corollary}~\ref{cor:sym-breg} are not satisfied we are not guaranteed that $0$ is not a solution to \eqref{eq:bilevel}. Indeed, $0$ is the optimal parameter in this case, indicating that the assumption of zero mean noise in {Corollary}~\ref{cor:sym-breg} is tight.
		}
	\end{figure}
	
	%%%%%%%%%%%%%%%%%%%%%%%%%%%%%%%%%%%%%%%%%%%%%%%%%%%%%%%%%%%%%%%%%%%%%
	%%%%%%%%%%%%%%%%%%%%%%%%%%%%%%%%%%%%%%%%%%%%%%%%%%%%%%%%%%%%%%%%%%%%%
	%%%%%%%%%%%%%%%%%%%%%%%%%%%%%%%%%%%%%%%%%%%%%%%%%%%%%%%%%%%%%%%%%%%%%
	\subsubsection{Predictive risk upper level}
	
	We now demonstrate the result of Corollary~\ref{cory:pointwise-deblur} where the upper level cost is the predictive risk \eqref{eq:bilevel-pred-risk-ul-og} without the expectation and $\op$ is invertible.
	In particular, we consider the same setup as above with forward operator $\op_2$ and Tikhonov regularization. 
	We plot the region for which $0$ is a solution and boundary for when the condition of Corollary~\ref{cory:pointwise-deblur} holds in Figure~\ref{fig:numeric-pred-small}.
	Since the upper level cost is different to the one considered in the previous numerics, the contour plot of the upper level cost looks very different.
	We see that the condition of Corollary~\ref{cory:pointwise-deblur} characterizes whether $0$ is a solution to \eqref{eq:bilevel-pred-risk-og} well in this setting.
	Despite not being the denoising setting, the region for which the condition of Corollary~\ref{cory:pointwise-deblur} is satisfied is similar in shape to the analogous region in Figure~\ref{fig:numeric-tikh-denoise} for which the condition of Corollary~\ref{cory:pointwise} is satisfied for the case $\op=I.$ This is likely because, due to the invertibility of $\op_2$, the predictive risk bilevel learning problem \eqref{eq:bilevel-pred-risk-og} is related to a certain denoising bilevel learning problem with the squared error upper level cost.

	\begin{figure}
		\centering
		\includegraphics[width=.47\textwidth]{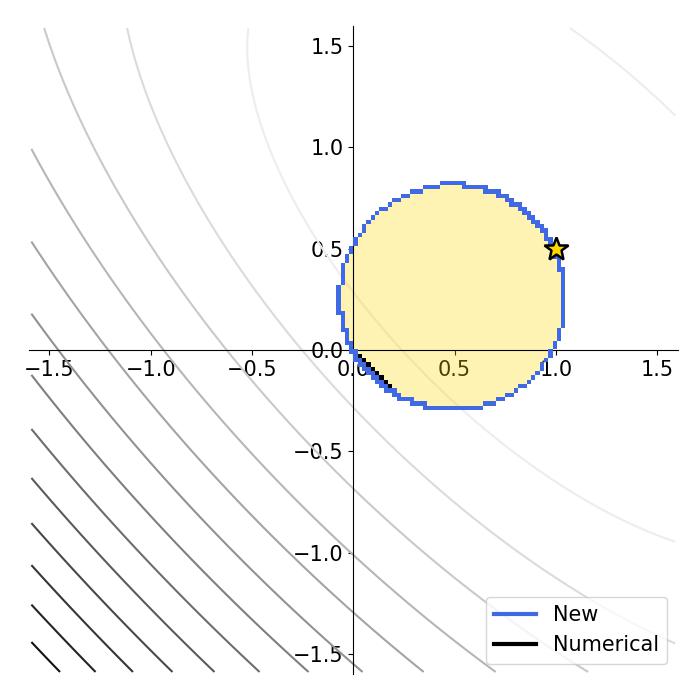}
		\caption{	Visualisation of Corollary~\ref{cory:pointwise-deblur}. Plots of the reconstruction space $\Omega$ for the non-trivial forward operator ($\op_2$) setting, Tikhonov regularization, and predictive risk upper level cost. 
			The ground truth $\xtrue=[1,0.5]^T$ is represented by a yellow star, and level sets of the upper level cost function are visible. 
			The region where $0$ is a solution to \eqref{eq:bilevel} is shaded yellow.} 
		\label{fig:numeric-pred-small}
	\end{figure}

	%%%%%%%%%%%%%%%%%%%%%%%%%%%%%%%%%%%%%%%%%%%%%%%%%%%%%%%%%%%%%%%%%%%%
	%%%%%%%%%%%%%%%%%%%%%%%%%%%%%%%%%%%%%%%%%%%%%%%%%%%%%%%%%%%%%%%%%%%%
	\subsection{High-dimensional problems}\label{subsec:numerics:largescale}
	
	In this subsection we consider two settings: a denoising scenario where we show that the old condition \eqref{eq:heuristic-reg-cond} can fail to capture $0$ not being a solution to \eqref{eq:bilevel}; and a deconvolution scenario where the existing theory is no longer applicable but we can instead employ both Theorem~\ref{thm:expec} and Theorem~\ref{thm:positivity-deblur} in the pointwise setting, that is, Corollary~\ref{cory:pointwise} and Corollary~\ref{cory:pointwise-deblur} respectively.

	\subsubsection{Denoising application}\label{subsec:numerics-large-denoise}
	We consider the pointwise denoising setting and provide two examples: one where both the new and old conditions are satisfied, and one where only the new condition is satisfied.  
	We consider the $128\times 128$ pixel ground truth image displayed in Figure~\ref{fig:large-denoise-xstar} and its additive Gaussian noise corrupted version, displayed in Figure~\ref{fig:large-denoise-y}, where the noise level,
	\begin{equation*}
		\eta :=  \frac{\norm{\op\xtrue - \ynoise}}{\norm{\op\xtrue}},
	\end{equation*}
	is chosen to be $\eta = 0.05.$
	
	\begin{figure}
		\centering
		\begin{subfigure}{.45\textwidth}
			\includegraphics[width=\textwidth]{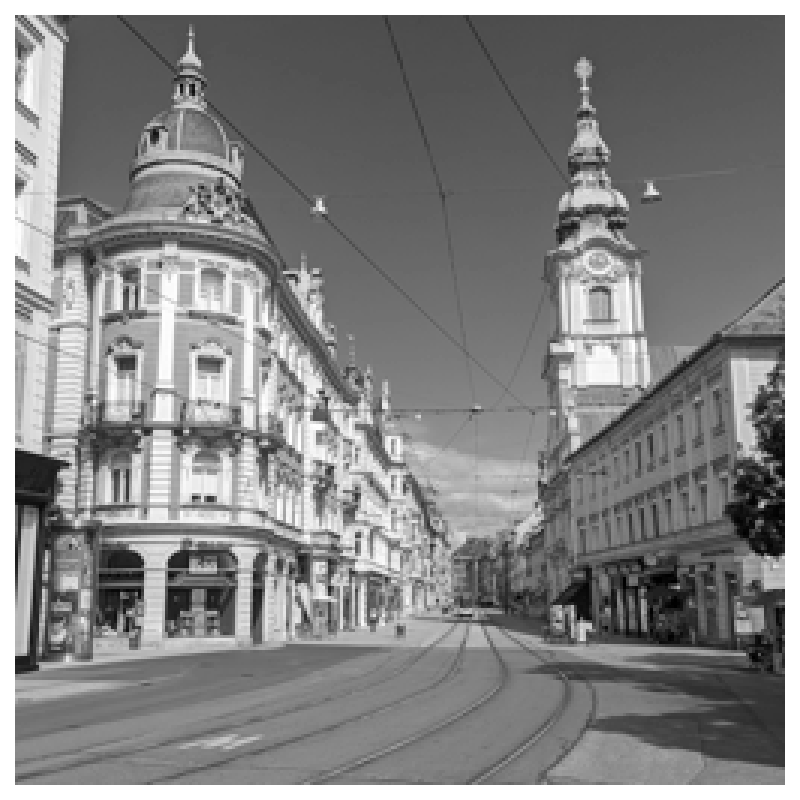}
			\caption{Ground truth $\xtrue$}\label{fig:large-denoise-xstar}
		\end{subfigure}
		\begin{subfigure}{.45\textwidth}
			\includegraphics[width=\textwidth]{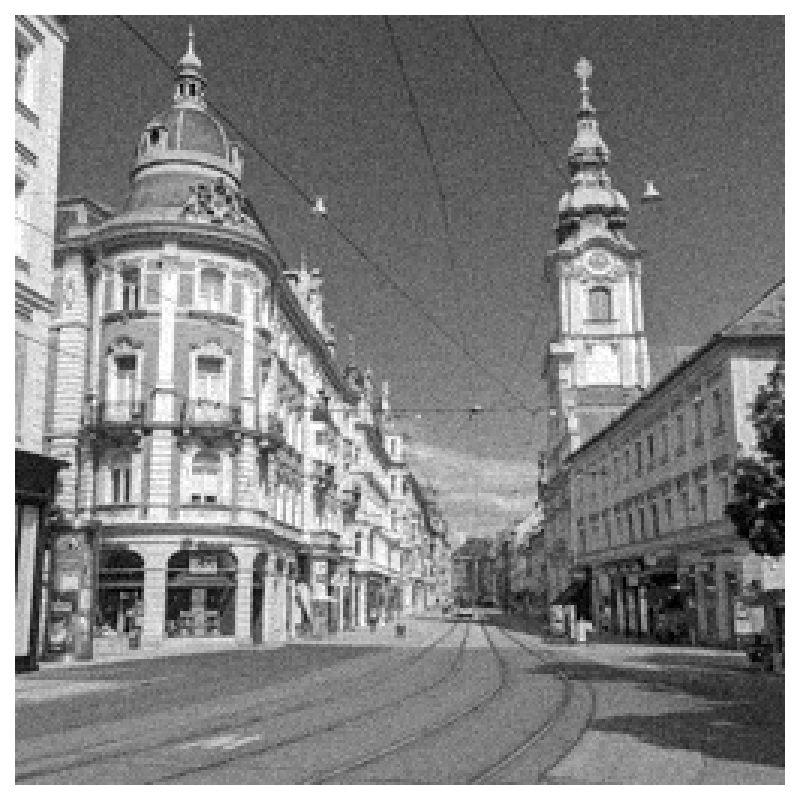}
			\caption{Noisy measurement $\ynoise$}\label{fig:large-denoise-y}
		\end{subfigure}
		\caption{ Relevant data for the large-scale denoising problem. The pixel values are clipped to $[0,255]$ where $0$ is black and $255$ is white. The test image was provided by one of the authors.}
	\end{figure}
	
	We consider both Huber norm and Huber TV regularization, given by
	\begin{equation*}
		\reg(x) = \sum_{i=1}^ n \mathrm{hub}_\gamma(x_i) \quad\text{and}\quad
		\reg(x) = \sum_{i=1}^ p \mathrm{hub}_\gamma([\nabla x]_i) \end{equation*}
	respectively, where $\nabla$  calculates both the horizontal and vertical gradient of $x$ and returns the vectorised concatenation of both results.
	We use the smoothing parameter $\gamma=0.01$ in both cases.
	For Huber norm and Huber TV regularization, we consider parameter spaces $[0,0.1]$ and \tcb{$[0,5]$} respectively, linearly discretised into 50 points. 
	
	\begin{figure}[ht!]
		\centering
		\begin{subfigure}{.45\textwidth}
			\includegraphics[width=\textwidth]{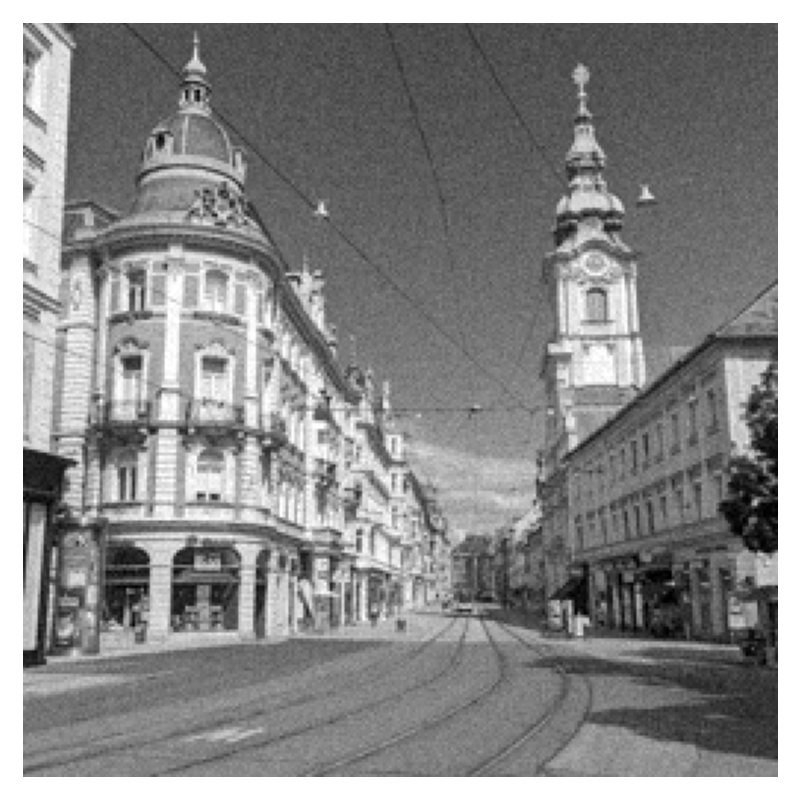}
			\caption{Huber norm reconstruction using the solution to \eqref{eq:bilevel}, $\alpha = 0.073.$ }
			\label{fig:large-denoise-hub-xa}
		\end{subfigure}
		\begin{subfigure}{.45\textwidth}
			\includegraphics[width=\textwidth]{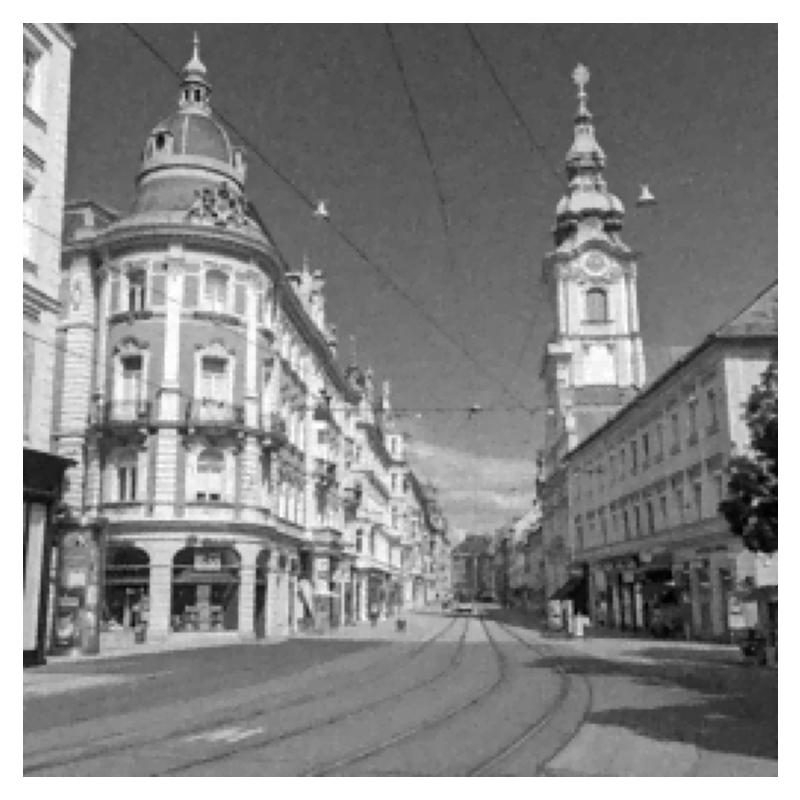}
			\caption{Huber TV reconstruction using the solution to \eqref{eq:bilevel}, $\alpha = 2.755.$}
			\label{fig:large-denoise-tv-xa}
		\end{subfigure}
		\begin{subfigure}{.45\textwidth}
			\includegraphics[width=\textwidth]{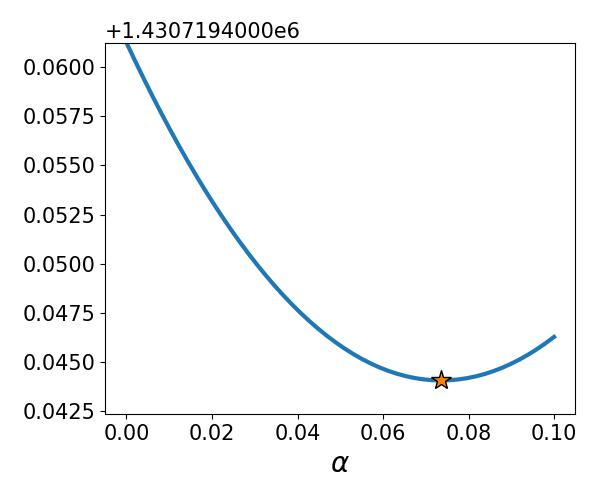}
			\caption{Upper level cost for Huber norm regularization}
			\label{fig:large-denoise-hub-ul}
		\end{subfigure}
		\begin{subfigure}{.45\textwidth}
			\includegraphics[width=\textwidth]{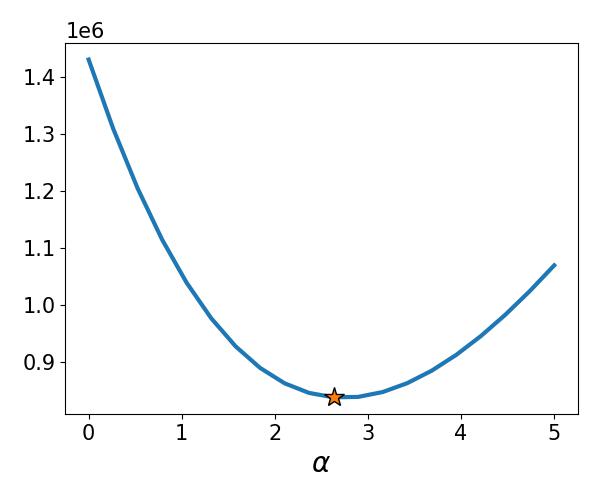}
			\caption{Upper level cost for Huber TV regualarization}
			\label{fig:large-denoise-tv-ul}
		\end{subfigure}
		\caption{Visualisation of Corollary~\ref{cory:pointwise}. Reconstructions and plots of the upper level cost function for the large scale denoising numerical experiment. The pixel values are clipped to $[0,255]$ where $0$ is black and $255$ is white.
			For Huber norm regularization, despite the old condition being violated, the optimal parameter is non-zero.}
	\end{figure}

	For the Huber norm, we get that \tcb{$\reg(\xtrue) = 8051925$} \tcb{to 7 significant figures} \tcb{(7~s.f.) and $\reg(\ynoise)=8051561$ (7~s.f.)} and so the old condition \eqref{eq:heuristic-reg-cond} is not satisfied. 
	However, in Figure~\ref{fig:large-denoise-hub-ul} we see that the optimal regularization parameter is non-zero, despite the associated reconstruction, displayed in Figure~\ref{fig:large-denoise-hub-xa}, looking similar to the noisy measurement.
	We remark that the differences in the upper level cost for the considered parameters are very small.
	The relevant inequality of Corollary~\ref{cory:pointwise} is satisfied, reading \tcb{$8051514 < 8051560$ (7~s.f.).} 
	Thus Corollary~\ref{cory:pointwise} successfully captures the fact that $0$ is not a solution to the bilevel learning problem, whereas the existing theory is inconclusive.
	
	For Huber TV, the assumptions of Corollary~\ref{cory:pointwise} and the old condition \eqref{eq:heuristic-reg-cond} are both satisfied. 
	In Figure~\ref{fig:large-denoise-tv-ul} we see that the optimal regularization parameter is non-zero, and the associated reconstruction, displayed in Figure~\ref{fig:large-denoise-tv-xa}, is an improvement upon on the noisy measurement.
	
	\subsubsection{Deconvolution application}\label{subsec:numerics-large-deblur}
	We now consider a pointwise deconvolution problem where the forward operator is a Gaussian blur with standard deviation \tcb{$1.3$.} 
	The ground truth image is the same as that of Section~\ref{subsec:numerics-large-denoise} and is displayed in Figure~\ref{fig:large-denoise-xstar}. 
	The observed measurement consists of the blurred ground truth corrupted by additive Gaussian noise of noiselevel $\eta=0.05$ and is displayed in Figure~\ref{fig:large-deblur-y}.
	The unregularized reconstruction $\xo$ is displayed in Figure~\ref{fig:large-deblur-xo} and is clearly dominated by noise.
	We consider the squared error bilevel problem \eqref{eq:bilevel} and also the predictive risk bilevel problem \eqref{eq:bilevel-pred-risk-og} in order to apply both Corollary~\ref{cory:pointwise} and Corollary~\ref{cory:pointwise-deblur}.
	We consider Huber TV regularization, and approximate the solution to \eqref{eq:bilevel} by considering the interval \tcb{$[0,3]$} linearly discretised into 50 points and select the parameter that achieves the smallest upper level cost.
	
	\begin{figure}
		\centering
		\begin{subfigure}{.45\textwidth}
			\includegraphics[width=\textwidth]{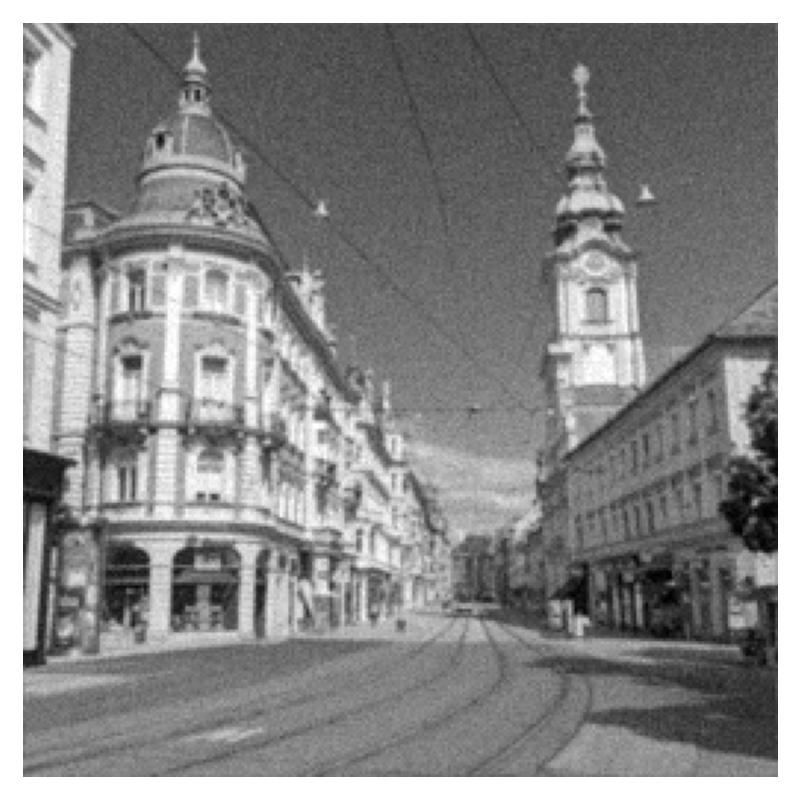}
			\caption{Corrupted blurry measurment $\ynoise$}\label{fig:large-deblur-y}
		\end{subfigure}
		\begin{subfigure}{.45\textwidth}
			\includegraphics[width=\textwidth]{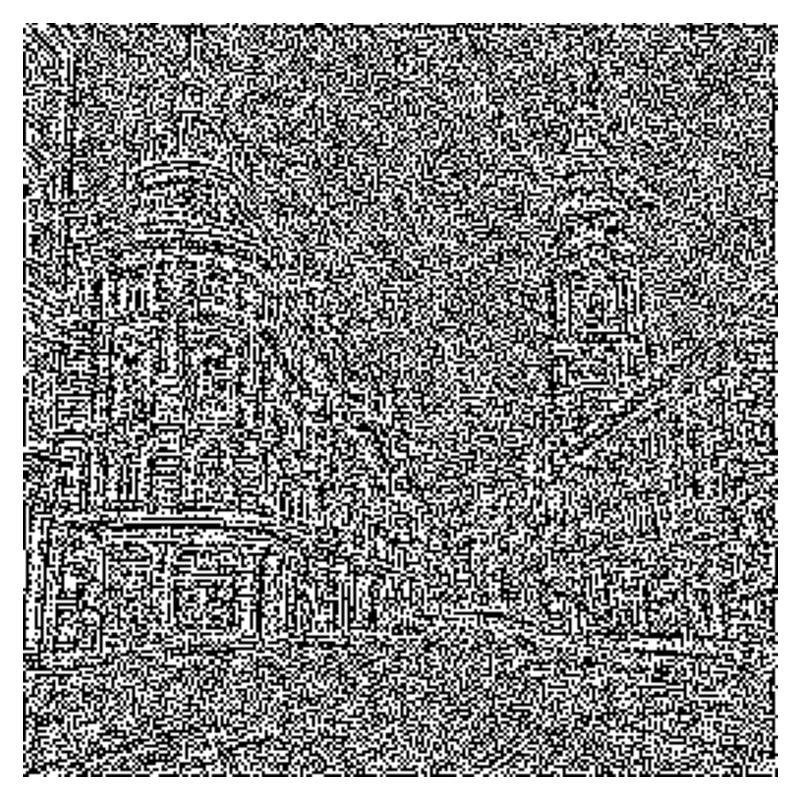}
			\caption{{Unregularized} reconstruction $\xo$}\label{fig:large-deblur-xo}
		\end{subfigure}
		\caption{Relevant data for the large-scale deconvolution problem. The pixel values are clipped to $[0,255]$ where $0$ is black and $255$ is white. We see that the unregularized reconstruction is completely dominated by noise.}
	\end{figure}

	\begin{figure}[ht!]
		\centering
		\begin{subfigure}{.45\textwidth}
			\includegraphics[width=\textwidth]{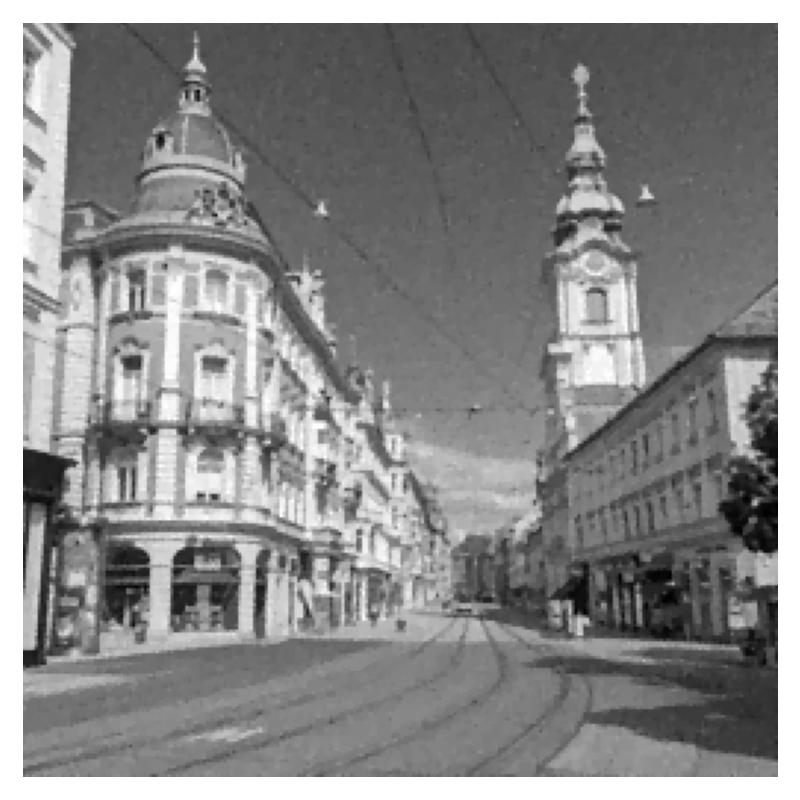}
			\caption{Huber TV reconstruction using the solution to \eqref{eq:bilevel}, $\alpha = 1.895$}\label{fig:large-deblur-mse-xa}
		\end{subfigure}
		\begin{subfigure}{.45\textwidth}
			\includegraphics[width=\textwidth]{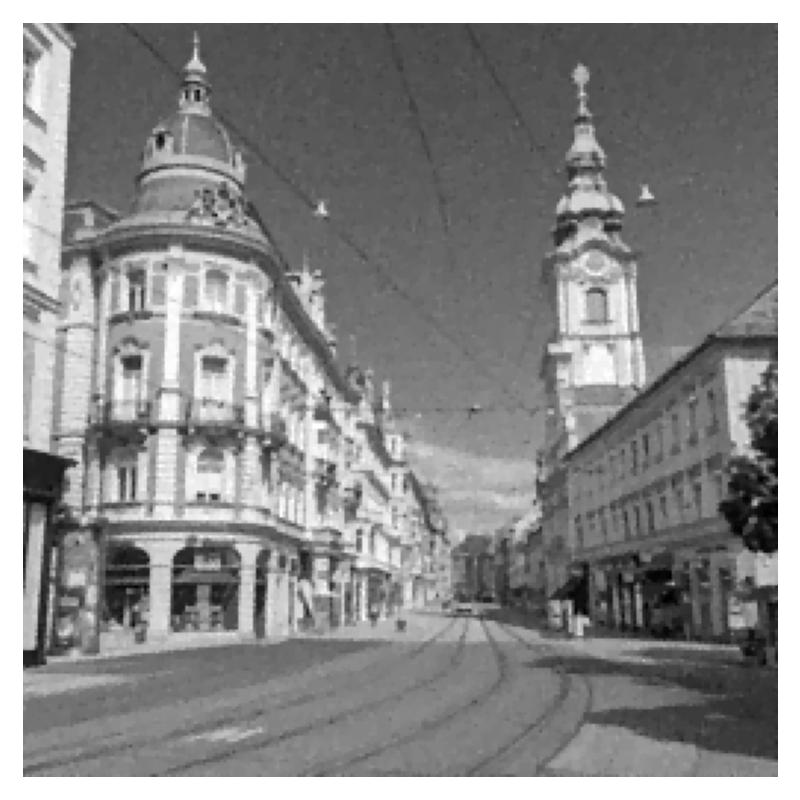}
			\caption{Huber TV reconstruction using the solution to \eqref{eq:bilevel-pred-risk-og}, $\alpha=1.898$ }\label{fig:large-deblur-pred-xa}
		\end{subfigure}
		\begin{subfigure}{.45\textwidth}
			\includegraphics[width=\textwidth]{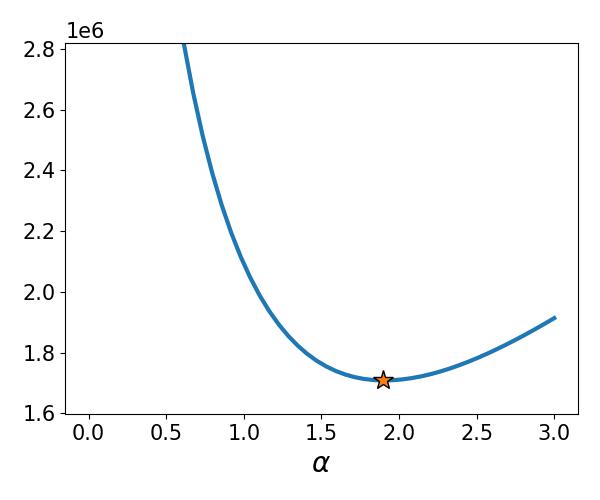}
			\caption{Squared error upper level cost \eqref{eq:ul} for Huber TV regularization}\label{fig:large-deblur-mse-ul}
		\end{subfigure}
		\begin{subfigure}{.45\textwidth}
			\includegraphics[width=\textwidth]{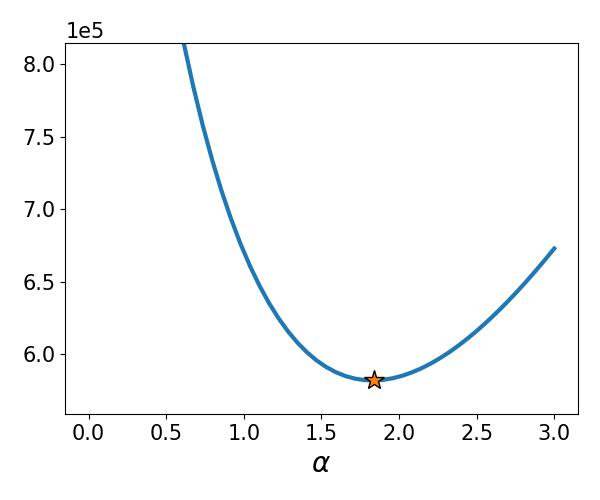}
			\caption{Predictive risk upper level cost \eqref{eq:bilevel-pred-risk-ul-og} for Huber TV regularization}\label{fig:large-deblur-pred-ul}
		\end{subfigure}
		\caption{Visualisation of Corollary~\ref{cory:pointwise} and Corollary~\ref{cory:pointwise-deblur}. Reconstructions, and plots of the upper level cost function for the large scale deblurring numerical experiment.
			The pixel values are clipped to $[0,255]$ where $0$ is black and $255$ is white.
			\tcb{In (c) and (d) we only display upper level cost evaluations that are finite and less than $10^{300}.$}
			Since we are not in the denoising setting, the old condition is not applicable but the presented theory succesfully identifies that $0$ is not a solution to the bilevel problem with both the squared error and predictive risk upper level cost.}
	\end{figure}

	We find that the inequalities of Corollary~\ref{cory:pointwise} and Corollary~\ref{cory:pointwise-deblur} are satisfied, with the larger side of the inequality being at least \tcb{150} orders of magnitude larger in both cases.
	This indicates that if the problem is ill-posed and the unregularized reconstruction $\xo$ is dominated by noise, then the relevant inequalities may be trivially satisfied in practice.

	Since both Corollary~\ref{cory:pointwise} and Corollary~\ref{cory:pointwise-deblur} are satisfied we are guaranteed that $0$ is not a solution to both the squared error and predictive risk bilevel learning problems, \eqref{eq:bilevel} and \eqref{eq:bilevel-pred-risk-og} respectively.
	Indeed, plots of the corresponding upper level cost functions are displayed in Figure~\ref{fig:large-deblur-mse-ul} and Figure~\ref{fig:large-deblur-pred-ul} where we see that both solutions are non-zero. 
	In this instance, both cost functions yield a similar optimal parameter, with the associated reconstructions displayed in Figure~\ref{fig:large-deblur-mse-xa} and Figure~\ref{fig:large-deblur-pred-xa} respectively.
	
	%%%%%%%%%%%%%%%%%%%%%%%%%%%%%%%%%%%%%%%%%%%%%%%%%%%%%%%%%%%%%%%%%%%%%%%%%%%%  
	\section{Conclusion}\label{sec:future}
	In this work we contributed to the fundamental understanding of bilevel learning as a mathematically sound regularization parameter choice strategy.
	More precisely, we determined an upper bound of the upper right Dini derivative of the upper level cost function at $0$ in terms of  Bregman distances and evaluations of the regularizer. In demanding that this upper bound is strictly negative, we established  a new sufficient condition to guarantee positivity of solutions of the bilevel learning problem, applicable to settings with an injective forward operator.
	In addition to this, an extension to the predictive risk cost function was made for an invertible forward operator.
	Furthermore, we showed that in most common denoising problems our condition will always be satisfied and we are guaranteed that $0$ is not an optimal regularization parameter.
	We have shown both analytically and empirically that the presented results characterize positivity well, and are an improvement on a condition commonly used in existing theory.

	\section*{Acknowledgements} % Acknowledgements
		MJE acknowledges support from EPSRC (EP/S026045/1, EP/T026693/1, EP/V026259/1) and the Leverhulme Trust (ECF-2019-478).
		The work of SG was partially supported by EPSRC under grant EP/T001593/1. 
		SJS is supported by a scholarship from the EPSRC Centre for Doctoral Training in Statistical Applied Mathematics at Bath (SAMBa), under the project EP/S022945/1.

\end{document}